\documentclass[12pt]{amsart}

\usepackage[all]{xypic}
\usepackage{tikz}
\usetikzlibrary{arrows} 
\usetikzlibrary{decorations.markings}
\usepackage{graphicx}
\usepackage{bm}
\usepackage{epsf}
\usepackage{verbatim} 
\usepackage{colonequals} 
\usepackage{amsmath}
\usepackage{amsfonts}
\usepackage{amssymb}
\usepackage{mathrsfs}
\usepackage{amsthm}
\usepackage{newlfont}

\newtheorem{thm}{Theorem}
\newtheorem{prop}[thm]{Proposition}
\newtheorem{lem}[thm]{Lemma}
\newtheorem{cor}[thm]{Corollary}

\theoremstyle{definition}
\newtheorem{defn}[thm]{Definition}

\newtheorem{example}[thm]{Example}

\theoremstyle{remark}
\newtheorem{rem}[thm]{Remark}

\numberwithin{equation}{section}
\numberwithin{thm}{section}

\DeclareMathOperator{\rank}{rank}

\DeclareMathOperator{\Spec}{Spec}

\DeclareMathOperator{\Mat}{Mat}

\DeclareMathOperator{\Gal}{Gal}
\DeclareMathOperator{\GL}{GL}
\DeclareMathOperator{\PGL}{PGL}
\DeclareMathOperator{\SL}{SL}
\DeclareMathOperator{\PSL}{PSL}

\DeclareMathOperator{\Stab}{Stab}

\DeclareMathOperator{\Nr}{Nr}

\newcommand{\tor}{\mathrm{tor}}

\newcommand{\dd}{\mathrm{d}}

\newcommand{\me}{\mathrm{e}} 

\newcommand{\fn}{\mathfrak{n}}

\newcommand{\fp}{\mathfrak{p}}

\newcommand{\cC}{\mathcal{C}}

\newcommand{\cF}{\mathcal{F}}

\newcommand{\cI}{\mathcal{I}}

\newcommand{\cK}{\mathcal{K}}

\newcommand{\cO}{\mathcal{O}}

\newcommand{\cT}{\mathcal{T}}

\newcommand{\sT}{\mathscr{T}}

\newcommand{\A}{\mathbb{A}}

\newcommand{\C}{\mathbb{C}}

\newcommand{\F}{\mathbb{F}}

\newcommand{\p}{\mathbb{P}}

\newcommand{\T}{\mathbb{T}}

\newcommand{\Z}{\mathbb{Z}}

\newcommand{\eps}{\varepsilon}
\newcommand{\G}{\Gamma}
\newcommand{\To}{\longrightarrow}
\newcommand{\bs}{\setminus}
\newcommand{\Fi}{F_\infty}
\newcommand{\Ci}{\C_\infty}

\newcommand{\oG}{\overline{\Gamma}}

\newcommand{\La}{\Lambda}
\newcommand{\la}{\lambda}

\newcommand{\twist}[1]{#1\!\left\{\tau\right\}}

\newcommand{\atwist}[2]{#1\langle#2\rangle}

\newcommand{\ls}[2]{#1\!\left(\mkern-4mu\left(#2\right)\mkern-4mu\right)}

\newcommand{\Mod}[1]{\ (\mathrm{mod}\ #1)}
\newcommand{\abs}[1]{\left|#1\right|}

\newcommand{\longhookrightarrow}{\lhook\joinrel\longrightarrow}

\begin{document}
\title{On Drinfeld modular curves for SL(2)}
\author{Jesse Franklin}
\address{Department of Mathematics, Salt Lake Community College, Salt Lake City, United States of America}
\email{jfranklin1185@gmail.com}

\author{Sheng-Yang Kevin Ho}
\address{Department of Mathematics, Pennsylvania State University, University Park, Pennsylvania, United States of America}
\email{kevinho@psu.edu}

\author{Mihran Papikian}
\address{Department of Mathematics, Pennsylvania State University, University Park, Pennsylvania, United States of America}
\email{papikian@psu.edu}

\thanks{The third author was supported in part by a Collaboration Grant for Mathematicians from the Simons Foundation, Award No. 637364.}


\begin{abstract}
	We study the Drinfeld modular curves arising from the Hecke congruence subgroups of $\SL_2(\F_q[T])$. Using a combinatorial 
	method of Gekeler and Nonnengardt, we obtain a genus formula 
	for these curves. In cases when the genus is one, we compute the Weierstrass equation of the corresponding curve. 
\end{abstract}

\maketitle


\section{Introduction} 

Let $\F_q$ be the finite field with $q$ elements and let $A=\F_q[T]$ be the 
ring of polynomials in indeterminate $T$ with coefficients in $\F_q$. Let $F=\F_q(T)$ be the fraction field of $A$. 
The degree function $\deg=\deg_T\colon A\to \Z_{\geq 0}\cup \{-\infty\}$, which assigns to $0\neq a\in A$ its degree 
as a polynomial in $T$ and $\deg_T(0)=-\infty$, extends to $F$ by $\deg(a/b)=\deg(a)-\deg(b)$. 
The map $-\deg$ is a valuation on $F$; the corresponding place of $F$ is usually denoted by $\infty$. 
Let $\abs{\cdot}$ denote the corresponding absolute value on $F$ normalized by $\abs{T}=q$. 
The completion $\Fi$ of $F$ with respect to this absolute value is isomorphic to field $\ls{\F_q}{1/T}$ 
of Laurent series in $1/T$. Finally, let $\Ci$ be the completion of an algebraic closure of $\Fi$. The absolute value 
$\abs{\cdot}$ has a unique extension, also denoted by $\abs{\cdot}$, to $\Ci$. 

Through the well-known relation between $A$-lattices of rank $2$ in $\Ci$ and Drinfeld $A$-modules of rank $2$ over $\Ci$, cf. \cite{Drinfeld}, 
the group $\GL_2(A)$ comes to play a role in the theory of Drinfeld modules and Drinfeld modular forms similar 
to the group $\SL_2(\Z)$ in the theory of elliptic curves over $\C$ and classical modular forms. 
In particular, the orbits of the action of $\GL_2(A)$ on the Drinfeld upper half-plane $\Omega\colonequals \Ci-\Fi$
 are in bijection with the isomorphism classes of rank $2$ Drinfeld modules over $\Ci$. 
 
 By considering congrunce subgroups of $\GL_2(A)$, one obtains Drinfeld modular curves of higher genus. 
 These modular curves possess rich arithmetic theory. An important example of a congruence group  is 
 the Hecke congruence group $\G_0(\fn)$ of level $\fn\in A$ defined by 
 $$
 \G_0(\fn)=\left\{ \begin{pmatrix} a & b \\ c & d\end{pmatrix}\in \GL_2(A)\ :\  \fn\mid c \right\}. 
 $$  
  The quotient $Y_0(\fn)\colonequals \G_0(\fn)\bs \Omega$ is an affine 
 curve with a canonical model over $F$; let $X_0(\fn)$ be the projective completion of $Y_0(\fn)$. The analogue of the modularity 
 theorem is known in this context by the results of Drinfeld and Deligne (cf. \cite{GR}): If $E$ is an elliptic curve over $F$ 
 with split multiplicative reduction at $\infty$ and conductor $\fn\cdot \infty$, then there is a non-constant morphism $X_0(\fn)\to E$. 
 
 In this article, we are interested in the congruence groups 
 $$
 \G_0^1(\fn)\colonequals \G_0(\fn)\cap \SL_2(A)
 $$
 and the corresponding Drinfeld modular curves $X_0^1(\fn)$. One reason why congruence subgroups of $\SL_2(A)$ are natural and interesting 
 is given in Lemma \ref{lem2.2}, where it is shown that the direct sum of modular forms on $\G_0(\fn)$ of fixed weight but varying type is isomorphic 
 to the space of modular forms on $ \G_0^1(\fn)$. The group $\SL_2(A)$ is also studied in Serre's book \cite[$\S$II.2]{Serre-Trees} 
 and in Gerritzen and van der Put's book \cite[Ch. X]{GvdP}. We also note that $X_0^1(\fn)=X_0(\fn)$ when $q$ is even, and there 
 is a natural morphism $X_0^1(\fn)\to X_0(\fn)$ of degree $2$ when $q$ is odd. 
 
 \vspace{0.1in}
 
 The main results of this article are the following:
 
 \begin{itemize}
 	\item In Section \ref{sGF}, we compute the genus of $X_0^1(\fn)$ following the combinatorial algorithm of Gekeler and Nonnengardt \cite{GN}.  
 	The resulting formula depends on the degrees and multiplicities of prime factors in the prime decomposition of $\fn$. 
 	When $\fn$ is prime of odd degree, this formula appears in \cite[10.13.3]{GvdP}, where it is deduced by geometric methods. 
 	
 	\item Assume $q$ is odd. From the genus formula, we observe that $X_0^1(\fn)$ has genus $1$ if and only if $\deg(\fn)=2$ and $\fn$ is square-free; 
 	more explicitly, either $\fn=T(T+1)$ (up to an automorphism of $F$) or $\fn$ is irreducible of degree $2$. In these cases, in  Section \ref{sEquations}, we deduce the 
 	Weierstrass equation of $X_0^1(\fn)$. 
 	
 \end{itemize}

 The fact that $X_0^1(\fn)$ is an elliptic curve when $q$ is odd and $\fn$ is square-free of degree $2$ is somewhat remarkable. 
 Indeed, $X_0(\fn)$ is an elliptic curve only in two cases, namely when $q=2$ and $\fn$ is either $T^2(T+1)$ or $T^3$ (up to an automorphism of $F$); 
 cf. \cite[Cor. 2.20]{GN}. The corresponding Weierstrass equations in these cases are given in (9.7.2) and (9.7.3) of \cite{GR}.  
 In addition, among the other modular curves $X_1^*(\fn)$, $X_1(\fn)$, $X(\fn)$ considered in \cite{GN}, the only elliptic curve is $X_1(\fn)$ for 
 $q=3$ and $\fn=T(T+1)$. 
 
 \subsection*{Acknowledgements} 
 The authors thank John Voight for providing them with a \textsc{Magma} program for computing the equation of $X_0^1(T(T+1))$ for $q=3$. The 
 parametrizations in $\S$\ref{ssX01T(T+1)} were guessed from the outputs of that program. The third author also thanks 
 Gunther Cornelissen, Andreas Schweizer  and Douglas Ulmer 
 for helpful communications related to the topics of this paper. 
 
 
 \section{Preliminaries}\label{sPrelim}  
 
 In addition to the notation in the introduction, we will use the following notation and conventions. 
Each nonzero ideal $\fn\lhd A$ has a unique monic generator, which, by abuse of notation, we will also denote by $\fn$. 
It will always be clear from the context whether $\fn$ denotes an ideal or its monic generator. We call a nonzero prime 
ideal $\fp\lhd A$ a \textit{prime} of $A$; the primes of $A$ are in bijection with the set of monic irreducible polynomials 
of $A$ of positive degree. 


\subsection{Drinfeld modules}\label{ssDM} An \textit{$A$-lattice} $\La\subset \Ci$ of rank $r\geq 1$ is an $A$-submodule of the form 
$\La=A\omega_1+\cdots +A\omega_r$, where $\omega_1, \dots, \omega_r\in \Ci$ are linearly independent over $\Fi$. To such a 
lattice we associate its \textit{exponential function} 
$$
e_\La(x)=x\prod_{0\neq \la\in \La} \left(1-\frac{x}{\la}\right). 
$$
The function $e_\La(x)$ is everywhere convergent on $\Ci$ and gives a surjective $\F_q$-linear map $e_\La\colon \Ci\to \Ci$. Let 
$\atwist{\Ci}{x}=\{a_0x+a_1x^q+\cdots+a_n x^{q^n}\mid n\geq 0, a_0, \dots, a_n\in \Ci\}$ be the noncommutative ring 
of $\F_q$-linear polynomials with usual addition of polynomials but where multiplication is defined via the composition of polynomials. For each 
$a\in A$, there is $\phi_a^\La(x)\in \atwist{\Ci}{x}$ such that $\deg_x \phi_a^\La(x)=q^{r\deg_T(a)}$, $\frac{\dd}{\dd x}\phi_a^\La(x)=a$, and 
$$
e_\La(a x)= \phi_a^\La(e_\La(x)). 
$$
The map $\phi^\La\colon A\to \atwist{\Ci}{x}$, $a\mapsto \phi_a^\La(x)$, is an $\F_q$-algebra homomorphism called the \textit{Drinfeld module} 
of rank $r$ associated to $\La$. Conversely, if $\phi\colon A\to \atwist{\Ci}{x}$, $a\mapsto \phi_a(x)$, is an $\F_q$-algebra homomorphism defined by 
$\phi_T(x)=Tx+g_1x^q+\cdots +g_r x^{q^r}$ with $g_r\neq 0$, then there is a unique $A$-lattice $\La$ of rank $r$ such that $\phi=\phi^\La$.  
Via $\phi^\La$, $\Ci$ acquires a new $A$-module structure $a\circ z=\phi^\La_a(z)$ denoted ${^{\phi^\La}}\Ci$. Thus, there is an exact sequence of $A$-modules 
\begin{equation}\label{eqDMUnif}
0\To \La \To \Ci \overset{e_\La}{\To} {^{\phi^\La}}\Ci\To 0. 
\end{equation}
Two Drinfeld modules $\phi$ and $\psi$ are \textit{isomorphic} if there is $c\in \Ci^\times$ 
such that $\phi_a(cx)=c\psi_a(x)$ for all $a\in A$. One shows that $\phi^{\La_1}$ and $\phi^{\La_2}$ are isomorphic if and only if $\La_1=c\La_2$ for some 
$c\in \Ci^\times$, i.e., if and only if the corresponding lattices are homothetic. 
We will be primarily interested in Drinfeld modules of rank $2$. In this case, every lattice is homothetic to a lattice of the form $Az+A$ for some $z\in \Omega$. 
It is not hard to show that the isomorphism classes of Drinfeld modules of rank $2$ are in bijection with the orbits of $\GL_2(A)$  acting on $\Omega=\Ci-\Fi$ 
via linear fractional transformations 
$$\begin{pmatrix} a & b \\ c & d\end{pmatrix} z=\frac{az+b}{cz+d}. $$ 
Detailed proofs of the above statements can be found in \cite[Ch. 5]{PapikianGTM}.


\subsection{Modular curves}

 The center $Z( \GL_2(A))\cong \F_q^\times$ acts trivially on $\Omega$, so  
 the action of $\GL_2(A)$ on $\Omega$ factors through $\PGL_2(A)=\GL_2(A)/\F_q^\times$. Similarly, the action of 
 $\SL_2(A)$ factors through $\PSL_2(A)=\SL_2(A)/\{\pm 1\}$. Let $\oG_0(\fn)$ (resp. $\oG_0^1(\fn)$)  
 denote the image of $\G_0(\fn)$ in $\PGL_2(A)$ (resp. the image of $\G_0^1(\fn)$ in $\PSL_2(A)$). 
 There is a commutative diagram with exact rows
$$
\xymatrix{ 1 \ar[r] & \G_0^1(\fn) \ar[r] \ar[d]^-{\mod \pm 1} & \G_0(\fn) 
	\ar[d]^-{\mod\F_q^\times} \ar[r]^-{\det}
	& \F_q^\times \ar[d]^-{\mod (\F_q^\times)^2} \ar[r] & 1 \\
	1 \ar[r] & \oG_0^1(\fn)\ar[r] & \oG_0(\fn) \ar[r] & \F_q^\times/(\F_q^\times)^2\ar[r] & 1.}
$$
Hence
\begin{equation}\label{eqG:G^1}
[\oG_0(\fn): \oG_0^1(\fn)] =\begin{cases} 2 & \text{if $q$ is odd}; \\ 
1 & \text{if $q$ is even}\end{cases}
\end{equation}

The set $\Omega=\Ci-\Fi$ has a natural structure of a smooth rigid-analytic space; cf. \cite[$\S$1]{GR}. 
The quotients $Y_0(\fn)=\G_0(\fn)\bs \Omega$ and $Y_0^1(\fn)=\G_0^1(\fn)\bs \Omega$ are the underlying 
analytic spaces of smooth affine algebraic curves defined over $\Fi$; cf. \cite[$\S$2]{GR}.  Denote by $X_0(\fn)$ (resp. $X_0^1(\fn)$) 
the projective closure of $Y_0(\fn)$ (resp. $Y_0^1(\fn)$). From \eqref{eqG:G^1} it follows that $X_0(\fn)=X_0^1(\fn)$ if $q$ is even, 
and there is a natural morphism $X_0^1(\fn)\to X_0(\fn)$ of degree $2$ if $q$ is odd.  The set of points $X_0(\fn)-Y_0(\fn)$ are 
the \textit{cusps} of $X_0(\fn)$; the cusps of $X_0^1(\fn)$ are defined similarly. 

For $0\neq \fn\in A$ and a Drinfeld module $\phi$ of rank $2$ over $\Ci$, let $\phi[\fn]$ be the set of zeros of $\phi_{\fn}(x)$. This set  
$\phi[\fn]$ is an $A$-submodule of ${^\phi}\Ci$ and it follows from \eqref{eqDMUnif} that $\phi[\fn]\cong \La/\fn\La\cong (A/\fn)^2$; note that 
$\phi[\fn]$ depends only on the ideal generated by $\fn$. A cyclic $\fn$-submodule of $\phi$ is an $A$-submodule of $\phi[\fn]$ isomorphic to $A/\fn$. 
The curve $Y_0(\fn)$ classifies the isomorphism classes of pairs $(\phi, \cC_\fn)$, where $\phi$ is a Drinfeld module of rank $2$ and $\cC_\fn$ 
is a cyclic $\fn$-submodule. A moduli interpretation for $Y_0^1(\fn)$ can be deduced from \cite[Thm. 5.2]{Breuer-h}. If $q$ is odd, then 
$Y_0^1(\fn)$ parametrizes isomorphism classes of Drinfeld modules with cyclic $\fn$-submodules along with $(\F_q^\times)^2$-classes of 
$T$-torsion points on their determinant modules. We refer to \cite[$\S$5]{Breuer-h} for the definition of this latter concept. 

\begin{rem}
\label{r: Breuer-h}
	In \cite{Breuer-h}, Breuer 
	considers the group 
	$$\G_0^+(\fn)=\{\gamma\in \Gamma_0(\fn)\mid \det(\gamma)\in (\F_q^\times)^2\}
	$$ 
	and the quotient $Y_0^+(\fn)=\G_0^+(\fn)\bs \Omega$. It is easy to check that the image of $\G_0^+(\fn)$ in $\PGL_2(A)$ coincides with $\oG_0^1(\fn)$, so 
	$Y_0^+(\fn)=Y_0^1(\fn)$.
\end{rem}

\subsection{Modular forms}\label{ssDMF}  Our main references for this subsection are \cite{Goss-ModForms} and \cite{Gekeler-Coeff}. 
From now 
on, unless indicated otherwise, \textbf{we assume that $q$ is odd}.  

Let $\G$ be a congruence subgroup of $\GL_2(A)$, i.e., $\G(\fn)\subseteq \G\subseteq \GL_2(A)$ for some $\fn\neq 0$, where $\G(\fn)$ 
is the (normal) subgroup of $\GL_2(A)$ consisting of matrices congruent to the identity matrix modulo $\fn$. 
A \textit{Drinfeld modular form} for $\G$ of weight $k\in \Z_{\geq 0}$ and type $m\in \Z/(q-1)\Z$ is a holomorphic function 
$$
f\colon \Omega\To \Ci
$$
such that 
\begin{enumerate}
	\item[(i)] $f(\gamma z) =(\det \gamma)^{-m} (cz+d)^k f(z)$ for all $\gamma = \begin{pmatrix}
	    a & b \\ c & d
	\end{pmatrix}\in \G$, and 
	\item[(ii)] $f(z)$ is holomorphic at the cusps of $\G$. 
\end{enumerate}
Denote the space of such functions by $M_{k, m}(\G)$. 

We explain condition (ii). Because $\G$ is a congruence group, it contains the subgroup $U_b=\begin{pmatrix} 1 & bA \\ 0 & 1\end{pmatrix}$ for 
some nonzero $b\in A$. Condition (i) implies that $f(z+b)=f(z)$, which itself implies that $f(z)$ can be expanded as 
$$
f(z)=\sum_{n\in \Z} a_n (1/e_{bA}(z))^n, \qquad a_n\in \Ci, 
$$
assuming $\Im(z)\colonequals \inf_{\alpha\in \Fi}\abs{z-\alpha}\gg 0$.  We say that $f(z)$ is \textit{holomorphic at the cusp $\infty$} if in the above expansion $a_n=0$ for all $n<0$ (this vanishing of coefficients with negative indices does not depend on the choice of $b$). 
Next, for $g=\begin{pmatrix} a & b \\ c & d\end{pmatrix}\in \GL_2(A)$, put $f|_g (z) = (\det(g))^m (cz+d)^{-k}f(g z)$. This 
$f|_g$ satisfies (i) for $\gamma \in g^{-1}\G g$, which is again a congruence group. Condition (ii) means that $f|_g$ is holomorphic at $\infty$ for all $g\in \GL_2(A)$. 
(Note that $f|_g=f$ for $g\in \G$, so for this last condition to hold it suffices that $f|_g$ is holomorphic at $\infty$ for left coset representatives of $\G$ in $\GL_2(A)$.)

Now we specialize to $\G=\G_0(\fn)$. Since $\G_0(\fn)$ 
contains the group of scalar matrices $\begin{pmatrix} \alpha & 0 \\ 0 & \alpha\end{pmatrix}$, $\alpha\in \F_q^\times$, applying 
condition (i) to such matrices one concludes that if $M_{k, m}(\G_0(\fn))\neq 0$, then $k\equiv 2m\Mod{q-1}$. Hence, if  $M_{k, m}(\G_0(\fn))\neq 0$, then $k$ is necessarily even and $m=k/2$ or $m=k/2+(q-1)/2$ modulo $q-1$. Next, a simple calculation shows that the differential $\dd z$ on $\Omega$ satisfies 
$$
\dd(\gamma z) = \frac{\det(\gamma)}{(cz+d)^2}\dd z \quad \text{for all }\gamma\in \GL_2(\Fi). 
$$
Hence, if $f(z)\in M_{2k, k}(\G_0(\fn))$, then $f(z)(\dd z)^k$ can be identified with a $k$-fold differential form on $X_0(\fn)$. 
These differential forms are holomorphic up to a divisor that accounts for the ramification in the covering $\Omega\to Y_0(\fn)$, 
and the ramification at the cusps; cf. \cite[pp. 51-53]{GekelerLNM}. Thus, knowing the genus of $X_0(\fn)$ and the aformentioned divisor, one easily computes the dimension 
of $M_{2k, k}(\G_0(\fn))$ for arbitrary $k$ using the Riemann-Roch theorem. On the other hand, modular forms in $M_{2k, k+\frac{q-1}{2}}(\G_0(\fn))$ do 
not correspond to differential forms on $X_0(\fn)$, so the previous strategy does not quite work for computing the dimension of $M_{2k, k+\frac{q-1}{2}}(\G_0(\fn))$. 
To circumvent this problem, observe that the modular forms for $\G_0^1(\fn)$ do not have type, or rather the type is always $0$, so 
every modular form in $M_{2k, 0}(\G_0^1(\fn))$ corresponds to a $k$-fold differential form on $X_0^1(\fn)$. Now one can compute the 
dimension of $M_{2k, k+\frac{q-1}{2}}(\G_0(\fn))$ by applying the Riemann-Roch theorem to $X_0^1(\fn)$, thanks to the following lemma: 

\begin{lem}\label{lem2.2} For any $k\geq 0$, 
	$$
	M_{2k, 0}(\G_0^1(\fn)) = M_{2k, k}(\G_0(\fn)) \oplus M_{2k, k+\frac{q-1}{2}}(\G_0(\fn)). 
	$$
\end{lem}
\begin{proof} The argument here is similar to the proof of Lemma 4.3.1 in \cite{Miyake}. 
	Since $\G_0^1(\fn)$ is normal in $\G_0(\fn)$, $\G_0(\fn)$ acts on $M_{2k, 0}(\G_0^1(\fn))$ by 
	$f\mapsto f|_\gamma$, $\gamma\in \G_0(\fn)$. This action induces a representation of $\G_0(\fn)/\G_0^1(\fn)$ 
	on $M_{2k, 0}(\G_0^1(\fn))$. Since $\G_0(\fn)/\G_0^1(\fn) \cong \F_q^\times$, all irreducible representations of $\G_0(\fn)/\G_0^1(\fn)$ are 
	induced by powers of the determinant. Therefore, we obtain 
	$$
	M_{2k, 0}(\G_0^1(\fn)) =\bigoplus_{m\in \Z/(q-1)\Z}  M_{2k, m}(\G_0(\fn))= M_{2k, k}(\G_0(\fn)) \oplus M_{2k, k+\frac{q-1}{2}}(\G_0(\fn)). 
	$$
\end{proof}

Important examples of modular forms arise as ``coefficient forms". For $z\in \Omega$, we define the rank-$2$ lattice 
$
\La_z=A+Az\subset \Ci$. 
Denote the Drinfeld module of rank $2$ associated to $\La_z$ by $\phi^z$. It is determined by 
\begin{equation}\label{eqCoeffForms}
\phi^z_T=Tx+g(z)x^q+\Delta(z)x^{q^2}. 
\end{equation}
Then the functions $g(z)$ and $\Delta(z)$ are modular forms for $\GL_2(A)$ of type $0$ and of weights $q-1$ and $q^2-1$, respectively. 
Goss proved in \cite{Goss-ModForms} that $g(z)$ and $\Delta(z)$ are algebraically independent over $\Ci$, and 
$$
\bigoplus_{k\geq 1} M_{k, 0}(\GL_2(A))=\Ci[g, \Delta]. 
$$
It turns out that $\Delta(z)$ possesses a $(q-1)$-th root in the ring of modular forms for $\GL_2(A)$. More precisely, Gekeler proved in \cite{Gekeler-Coeff} that there is $h(z)\in M_{q+1, 1}$ such that $h(z)^{q-1}=\Delta(z)$, and 
$$
\bigoplus_{\substack{k\geq 1\\ m\in \Z/(q-1)\Z}} M_{k, m}(\GL_2(A))=\Ci[g, h]. 
$$

The \textit{$j$-function} is $j(z)\colonequals g(z)^{q+1}/\Delta(z)$. This function is holomorphic on $\Omega$ but has a pole at the cusp $\infty$. 
Since $j(\gamma z)=j(z)$ for all $\gamma\in \GL_2(A)$, it defines a rational function on $X_0(1)$. 
In fact, $j(z)$ generates the field of rational functions on $X_0(1)\cong \p^1_{\Ci}$. Put 
$$
\sqrt{j}(z)\colonequals g^{(q+1)/2}/h^{(q-1)/2}. 
$$
It is easy to check that 
$$
\sqrt{j}(\gamma z) = (\det(\gamma))^{(q-1)/2} \sqrt{j}(z)  \quad \text{for all }\gamma\in \GL_2(A). 
$$
Thus, $\sqrt{j}(z)$ is a modular function for $\SL_2(A)$. 

\begin{lem}\label{lemFFX0}
	Let $\cF_{0, \fn}$ and $\cF_{0, \fn}^1$ be the fields of rational functions on $X_0(\fn)$ and $X_0^1(\fn)$, respectively. 
	Then $\cF_{0, \fn}^1=\cF_{0, \fn}(\sqrt{j})$. 
\end{lem}
\begin{proof}
	Since $\sqrt{j}$ is not a modular function for $\G_0(\fn)$ but its square is, we have $$[\cF_{0, \fn}(\sqrt{j}):\cF_{0, \fn}]=2.$$ 
	On the other hand, $[\cF_{0, \fn}^1:\cF_{0, \fn}]=2$ because there is a morphism $X_0^1(\fn)\to X_0(\fn)$ of degree $2$. 
	Since $\sqrt{j}\in \cF_{0, \fn}^1$, we conclude that $\cF_{0, \fn}^1=\cF_{0, \fn}(\sqrt{j})$. 
\end{proof}


\section{Genus formula for $X_0^1(\fn)$}\label{sGF} In this section we compute a genus formula for $X_0^1(\fn)$ following 
the combinatorial algorithm of Gekeler and Nonnengardt \cite{GN}. The algorithm itself actually computes the 
quotient of the Bruhat-Tits tree $\cT$ of $\PGL_2(\Fi)$ under the action of a congruence subgroup $\G\subseteq \GL_2(A)$. 
The fact that the genus of the projective compactification of $\G\bs \Omega$ is equal to the dimension of 
the homology group $H_1(\G\bs \cT, \C)$ is essentially due to Mumford \cite{MumfordCurves} (see also Reversat \cite{Reversat}). 

A genus formula for $X_0^1(\fn)$ can also be deduced from the known genus formula for $X_0(\fn)$ (see \cite[Thm. 2.17]{GN})
by analyzing the double cover $X_0^1(\fn)\to X_0(\fn)$ and applying the Riemann-Hurwitz formula. In this approach, one 
needs to compute the number of cusps and elliptic points that ramify in the covering. This calculation does not seem to be much easier 
than the combinatorial and group-theoretic calculations in this section. In addition, the graph $\G_0^1(\fn)\bs \cT$ contains 
other interesting arithmetic information about $X_0^1(\fn)$, such as the structure of 
the group of connected components of the N\'eron model of the Jacobian variety of $X_0^1(\fn)$ over $\Fi$, or 
the space of $\C$-valued automorphic forms for $\G_0^1(\fn)$. 

 
 \subsection{Preliminaries} Let $\cO_\infty=\F_q[\![1/T]\!]$ be the ring of integers of $\Fi=\ls{\F_q}{1/T}$. 
 The Bruhat-Tits tree of $\PGL_2(\Fi)$ is a $(q+1)$-regular tree whose vertices and edges 
 are given by 
 \begin{align*}
 	X(\cT) &=\GL_2(\Fi)/\GL_2(\cO_\infty) \cdot Z(\Fi) \\ 
 	Y(\cT) &=\GL_2(\Fi)/\cI \cdot Z(\Fi), 
 \end{align*}
 where $\cI$ is the Iwahori subgroup consisting of matrices $\begin{pmatrix} a& b\\ c & d\end{pmatrix}\in \GL_2(\cO_\infty)$ 
 with $c\in T^{-1}\cO_\infty$, and $Z$ is the center of $GL_2$. Equivalently, the vertices of $\cT$ are the homothety classes $[L]$ of rank-$2$ 
 $\cO_\infty$-lattices $L$ in $\Fi^2$, 
 with two vertices being adjacent if one can choose representative lattices $L\subset L'$ such that $L'/L\cong \F_q$; see \cite[p. 35]{GR}.

 For $i\in \Z$, let $v_i$ be the vertex represented by the matrix $\begin{pmatrix} T^i & 0 \\ 0 & 1\end{pmatrix}$. Denote 
\begin{align*}
	G_0 &=\SL_2(\F_q)\hookrightarrow \SL_2(A)\\ 
	G_i & = \left\{\begin{pmatrix} a & b \\ 0 & a^{-1} \end{pmatrix}\mid a\in \F_q^\times, \deg(b)\leq i \right\}, \quad i\geq 1.  
\end{align*}
For each $i\geq 0$, $G_i$ is the stabilizer of $v_i$ in $\SL_2(A)$ and $G_i\cap G_{i+1}$ is the stabilizer of the 
edge $e_i$ with origin $v_i$ and terminus $v_{i+1}$. Note that $G_i\cap G_{i+1}=G_i$ if $i\geq 1$.

By \cite[pp. 111-112]{Serre-Trees}, the subgraph formed 
of the $v_i$ and $e_i$ with $i\geq 0$ maps isomorphically onto the quotient graph $\SL_2(A)\bs \cT$, so this 
quotient graph is a half-line; see Figure \ref{Fig1}. 
\begin{figure}
	\begin{tikzpicture}[->, >=stealth', semithick, node distance=1.5cm, inner sep=.5mm, vertex/.style={circle, fill=black}]
		
		\node[vertex] (0) [label=below:$v_0$]{};
		\node[vertex] (1) [right of=0, label=below:$v_1$] {}; 
		\node[vertex] (2) [right of=1, label=below:$v_2$] {};
		\node[vertex] (3) [right of=2, label=below:$v_3$] {};
		\node[] (4) [right of=3] {};
		
		\path[]
		(0) edge  (1) (1) edge (2) (2) edge (3) (3) edge[dashed] (4);   
	\end{tikzpicture}
	\caption{$\SL_2(A)\bs \sT$}\label{Fig1}
\end{figure}

The Gekeler--Nonnengardt algorithm applied to our situation recovers the quotient $\G_0^1(\fn)\bs \cT$ by examining the covering 
$$
\pi\colon  \G_0^1(\fn)\bs \cT \To \SL_2(A)\bs \cT. 
$$
For $i\geq 0$, $v\in X(\G_0^1(\fn)\bs \cT)$, $e\in Y(\G_0^1(\fn)\bs \cT)$, put 
\begin{align*}
	X_i &\colonequals X_i(\G_0^1(\fn)\bs \cT)=\{v\in X(\G_0^1(\fn)\bs \cT)\mid \pi(v)=v_i\}\\ 
	Y_i &\colonequals Y_i(\G_0^1(\fn)\bs \cT)=\{e\in Y(\G_0^1(\fn)\bs \cT)\mid \pi(e)=e_i\}. 
\end{align*}
We have 
\begin{align*}
	X_i &\cong G_i\bs \SL_2(A)/\G_0^1(\fn)\\ 
	Y_i &\cong  (G_i\cap G_{i+1})\bs \SL_2(A)/\G_0^1(\fn), 
\end{align*}
and accordingly we call $X_i$ and $Y_i$ the vertices and edges of type $i$, respectively. 

\begin{lem} Let 
	$$
	\p^1(A/\fn)\colonequals \{(u:v)\mid u, v\in A/\fn, (A/\fn)u +(A/\fn)v=A/\fn\},
	$$
	where $(u: v)$ is the equivalence class of $(u, v)$ modulo $(A/\fn)^\times$. There is an isomorphism 
	\begin{align*}
		\SL_2(A)/\G_0^1(\fn) &\overset{\sim}{\To} \p^1(A/\fn)\\ 
		\begin{pmatrix} a & b \\ c & d\end{pmatrix} &\longmapsto (a:c)\mod \fn
		\end{align*}
		as $\SL_2(A)$-sets, where the action of $\SL_2(A)$ on $\p^1(A/\fn)$ is 
		$$
		\begin{pmatrix} a & b \\ c & d\end{pmatrix}  (u:v) =(au+bv : cu+dv). 
		$$
\end{lem}
\begin{proof}
	This map is surjective since given $(a:c)\in \p^1(A/\fn)$ and an arbitrary representative of it $\begin{pmatrix} a \\ c\end{pmatrix}\in \Mat_{2,1}(A)$, we can always choose 
	$\begin{pmatrix} b \\ d\end{pmatrix}\in \Mat_{2,1}(A)$ such that $\begin{pmatrix} a & b \\ c & d\end{pmatrix}\in \GL_2(A)$. Scaling $\begin{pmatrix} b \\ d\end{pmatrix}$  by an element of $\F_q^\times$, we can assume that $\det\begin{pmatrix} a & b \\ c & d\end{pmatrix}=1$. 
	
	Under the given map, the preimage of $(1:0)$ in $\SL_2(A)$ is $\G_0^1(\fn)$. Since $\SL_2(A)$ acts transitively on $\p^1(A/\fn)$, the preimage of every 
	point of $\p^1(A/\fn)$ is a coset of $\G_0^1(\fn)$. 
\end{proof}

Thus, one can compute $\G_0^1(\fn)\bs \cT$ in ``layers", where each layer is in bijection with the orbits of $G_i$ acting on $\p^1(A/\fn)$. 
Since $G_i$ acts on $\p^1(A/\fn)$ through its quotient modulo $\fn$, the orbits of $G_i$ acting on $\p^1(A/\fn)$ do not change once $i\geq d-1$, where $d\colonequals \deg(\fn)$. This implies that the subgraph of $\G_0^1(\fn)\bs \cT$ consisting of edges of type $\geq d-1$ is a disjoint union of half-lines (as in Figure \ref{Fig1}), 
called \textit{cusps} (see the appendix for some explicit examples). The number of cusps is the number of orbits of $G_{d-1}$ acting on $\p^1(A/\fn)$.

As is explained in \cite[(1.8)]{GN}, as far as the computation of the genus 
$$
g\left(\G_0^1(\fn)\bs \cT\right) \colonequals  \rank_\Z H_1\left(\G_0^1(\fn)\bs \cT, \Z\right )
$$
is concerned, only the 0-th layer and the cusps matter:
\begin{equation}\label{eqGenus}
	g\left(\G_0^1(\fn)\bs \cT\right)  = 1+\# Y_0 - \# X_0-\# X_{d-1}.
\end{equation}

\begin{rem}
	We caution the reader that we use the same notation $G_i$, $B=G_0\cap G_1$, $X_i$, $Y_i$, etc. as in \cite{GN}, but our objects  
	arise from $\SL_2$ rather than $\GL_2$. This introduces some subtle but important differences in the calculations. 
\end{rem}


\subsection{Prime power case}  The calculation of various orbits and stabilizers in the primary case is a crucial preliminary step 
for the calculation of $g\left(\G_0^1(\fn)\bs \cT\right)$. Thus, in this subsection we assume that $\fn=\fp^r$ is the $r$-th power of a prime 
$\fp$ for some $r\geq 1$. 

To determine $X_0$ and $Y_0$ we need to compute the number of orbits of $G\colonequals G_0=\SL_2(\F_q)$ and 
$$B\colonequals G_0\cap G_1 = \left\{\begin{pmatrix}  a & b \\ 0 & a^{-1}\end{pmatrix}\mid a\in \F_q^\times, b\in \F_q \right\}
$$
acting on  $\p^1(A/\fp^r)$. We also need to compute the number of cusps, which is the number of orbits of $G_{d-1}$ acting on $\p^1(A/\fp^r)$. 
These calculations will be carried out in three separate propositions. 

First, we make some observations about $\p^1(A/\fp^r)$. Let $z=(u:v)\in \p^1(A/\fp^r)$. 
Either $u$ or $v$ must be relatively prime to $\fp$, so must be a unit in $A/\fp^r$. Identifying the elements of $A/\fp^r$ 
with polynomials in $A$ of degree $\leq r\cdot \deg(\fp)-1$, either $v=0$ or  there is a well-defined $0\leq h\leq r-1$ 
such that  $v=w\fp^h$, where $w\in A$ has degree $\leq \deg(\fp)(r-h)-1$ and is relatively prime to $\fp$. The \textit{height} of $z$ is $h(z)=h$ if $v\neq 0$, 
and $h(z)=r$ if $v=0$. 

If $h(z)=0$, then $z=(u:1)$ for a unique $u\in A/\fp^r$ depending on $z$. If $1\leq h(z)\leq r-1$, then 
$z=(1:w\fp^h)$ for a unique $w$ as above. If $h(z)=r$, then $z=(1:0)$.  Therefore, the set $ \p^1(A/\fp^r)(h)$ 
of points on $ \p^1(A/\fp^r)$ of height $h$ has cardinality: 
\begin{equation} \label{eqOrdPh}
	\# \p^1(A/\fp^r)(h) = 
	\begin{cases}
		\abs{\fp}^r, &  \text{if $h=0$}; \\ 
		\abs{\fp}^{r-h}-\abs{\fp}^{r-h-1}, & \text{if $1\leq h\leq r-1$}; \\ 
		1, & \text{if $h=r$}.
	\end{cases}
\end{equation}	

Summing all these numbers  we get: 

\begin{lem}\label{lemOrdP-prime}
	$\# \p^1(A/\fp^r)=\sum_{h=0}^r 	\# \p^1(A/\fp^r)(h) =\abs{\fp}^{r-1}(\abs{\fp}+1)$. 
\end{lem}

Suppose $z=(u:v) \in \p^1(A/\fn)$ is 
fixed by some element $\begin{pmatrix} a & b \\ c & d\end{pmatrix}\neq \pm 1$ of $G$. 

\noindent\underline{\textit{Case 1}}: $h(z)=0$. \\ 
In this case, $z=(u:1)$ and the equation resulting from $\begin{pmatrix} a & b \\ c & d\end{pmatrix} (u:1)=(u:1)$ 
is 
$$
cu^2+(d-a)u-b=0. 
$$

\underline{\textit{Case 1.1}}: The polynomial $cx^2+(d-a)x-b\in \F_q[x]$ is irreducible and $c\neq 0$. \\ 
In this case, $u$ generates $\F_{q^2}$ over $\F_q$ and $\F_{q^2}=\F_q[u] \hookrightarrow A/\fp^r$, so $\deg(\fp)$ is even. 
To find the stabilizer of $(u:1)$, we may apply $\begin{pmatrix} 1 & (d-a)/2c\\ 0 & 1\end{pmatrix}$ to $(u:1)$, so as to assume
that $u$ is a root of $x^2=t$ for some non-square $t\in \F_q$. We can represent each 
element of $\F_{q^2}$ as $\alpha+\beta u$ for some $\alpha, \beta\in \F_q$.  The map 
\begin{align}\label{eqNonSplTor}
	\F_{q^2}^\times & \To \GL_2(\F_q)\\ 
\nonumber \alpha+\beta u& \longmapsto \begin{pmatrix} \alpha & \beta t\\ \beta & \alpha\end{pmatrix}
\end{align}
is an embedding into $\GL_2(\F_q)$. These are the matrices that fix $(u:1)$ in  $\GL_2(\F_q)$.
The matrices which are in $G$ are those for which $\alpha^2-t\beta^2=1$, or equivalently 
$$
\Nr_{\F_q^2/\F_q}(\alpha+\beta u)=1.
$$
Hence, 
$$\Stab_{G}(z)\cong (\F_{q^2}^\times)^1,
$$
where $(\F_{q^2}^\times)^1\colonequals \ker(\Nr_{\F_q^2/\F_q}\colon \F_{q^2}^\times\to \F_q^\times)$. 
Since $\Nr_{\F_q^2/\F_q}\colon \F_{q^2}^\times \to \F_q^\times$ 
is surjective, we conclude that  $\# (\F_{q^2}^\times)^1=q+1$.
The length of the orbit of $z$ is 
$$\#G/ \# \Stab_{G}(z)=q(q^2-1)/(q+1)=q(q-1)=\#  (\F_{q^2}-\F_q).$$
This implies that those $z$ that fall under this case form one orbit. 

\underline{\textit{Case 1.2}}: $c=0$. \\ 
This case is equivalent to $u\in \F_q$, i.e., $z\in \p^1(\F_q)\hookrightarrow \p^1(A/\fp^r)$. 
The stabilizer of $(0:1)$ in $G$ is $B$, so the 
length of its orbit is $q(q^2-1)/q(q-1)=q+1$. Thus, all elements of $\p^1(\F_q)$ are in one orbit. 

\underline{\textit{Case 1.3}}: $cx^2+(d-a)x-b$ is reducible in $\F_{q}[x]$ and $c\neq 0$. \\ 
In this case, $u\in A/\fp^r$ satisfies $(u-e)(u-e')=0$, with $e, e'\in \F_q$. 
Thus, $(u-e)$ is a zero-divisor, i.e., $u-e\in \fp A/\fp^r$. Note that $e-e'\in \fp A/\fp^r$, so $e=e'$. 
We conclude that $z$ is in the set 
$$
S_1=\left\{(u:1)\mid u\in A/\fp^r-\F_q, \exists e\in \F_q \text{ such that } (u-e)^2=0\right\}. 
$$
To determine the stabilizer, we may assume that $u^2=0$ by first acting by $\begin{pmatrix} 1 & e\\ 0 & 1\end{pmatrix}$ on $z$. Then 
$$
\Stab_G(z) = \left\{\begin{pmatrix} \alpha & 0\\ \gamma & \alpha\end{pmatrix} \mid  \alpha=\pm 1, \gamma\in \F_q\right\}. 
$$
Therefore, each orbit has length $q(q^2-1)/2q=(q^2-1)/2$. To compute 
$\# S_1$, take any element of $\fp^{\lfloor r/2\rfloor}A/\fp^r$, except $0$. For each $e$, we get $(\abs{\fp}^{\lfloor r/2\rfloor}-1)$ elements of $S_1$, so 
$\# S_1= q(\abs{\fp}^{\lfloor r/2\rfloor}-1)$. We conclude that the number of orbits is 
$$
2q(\abs{\fp}^{\lfloor r/2\rfloor}-1)/(q^2-1). 
$$

\noindent\underline{\textit{Case 2}}: $h(z)\geq 1$. \\ 
In this case, $u$ must be a unit, so $z = (1:v)$.  Similar to Case 1, we get 
$dv^2+(c-b)v-a=0$. But since $v\in \fp A/\fp^r$, this is possible if and only if $v^2=0$, $a=0$, and $c=d$. We conclude that 
$z$ is in the set 
$$
S_2=\{(1:v)\mid 0\neq v\in A/\fp^r, v^2=0\}. 
$$
By an argument similar to Case 1.3, we get $\# S_2= (\abs{\fp}^{\lfloor r/2\rfloor}-1)$ and the number of orbits is 
$$
2(\abs{\fp}^{\lfloor r/2\rfloor}-1)/(q^2-1). 
$$

We summarize our previous computations into a proposition: 

\begin{prop}\label{prop:OrbitsofG1} The points of $\p^1(A/\fp^r)$, with their $G$-stabilizers and orbits, are:
	\begin{itemize}
	\item[$($a$)$] 
		$\displaystyle{\begin{cases}
				z\in \p^1(\F_q)\hookrightarrow \p^1(A/\fp^r); \\
				\Stab_G(z) \cong B; \\
				\text{one orbit of length } (q+1). 
			\end{cases}}$
		
		\item[$($b$)$] 
		$\displaystyle{\begin{cases}
			\deg(\fp) \text{ is even and } z\in \left(\p^1(\F_{q^2})-\p^1(\F_q)\right)\hookrightarrow \p^1(A/\fp^r);\\ 
		\Stab_G(z)\cong (\F_{q^2}^\times)^1; \\ 
		\text{one orbit of length }q(q-1). 
			\end{cases}}$
		
	\item[$($c$)$] 
	$\displaystyle{\begin{cases}
			z\in S_1\cup S_2;\\
			\Stab_G(z)\cong \left\{\pm  \begin{pmatrix} 1 & b\\ 0 & 1\end{pmatrix} \mid b\in \F_q\right\}; \\  
		2(\abs{\fp}^{\lfloor r/2\rfloor}-1)/(q-1)\text{ orbits, each of length }(q^2-1)/2. 
		\end{cases}}$
		
	\item[$($d$)$] 
	$\displaystyle{\begin{cases}
			\text{All other points};\\ 
			\Stab_G(z)\cong \{\pm 1\};\\
			\frac{2}{q(q^2-1)}\left(\abs{\fp}^{r-1}(\abs{\fp}+1) - (q+1)(\abs{\fp}^{\lfloor r/2\rfloor}-1) -\begin{cases} q+1, &  
				\text{if $\deg(\fp)$ is odd} \\ q^2+1, & \text{if $\deg(\fp) $ is even}\end{cases}\right)\\ 
			\quad \text{ orbits, each of length }q(q^2-1)/2. 
		\end{cases}}$
	
	\end{itemize}
\end{prop}

\begin{prop}\label{prop:orbitsB}
	The $G$-orbits of Proposition \ref{prop:OrbitsofG1} split  as follows into $B$-orbits: 
	\begin{itemize}
		\item[(a)] Two $B$-orbits: $\{(1:0)\}$ and $\{(u:1)\mid u \in \F_q\}$.  
		\item[(b)] Two $B$-orbits, both of length $q(q-1)/2$;
		\item[(c)] Two $B$-orbits: one of length $(q-1)/2$ given by elements of height $\geq 1$, 
		and the other of length $q(q-1)/2$, given by elements of height $0$. 
		\item[(d)] $(q+1)$ $B$-orbits of length $q(q-1)/2$. 
	\end{itemize}
\end{prop}
\begin{proof}
	(a) This is clear since upper-triangular unipotent  matrices act transitively on the set $\{(u:1)\mid u \in \F_q\}$ and fix $(1:0)$. 
	
	(b) We have $\begin{pmatrix} a & b \\ 0 & a^{-1}\end{pmatrix} (u:1) = (a^2 u+ab: 1)$. Suppose $u$ generates 
	$\F_{q^2}^\times$ over $\F_q$, then every element of $\F_{q^2}$ can be uniquely written as $au+b$ with $a, b\in \F_q$. 
	If we consider only $a^2 u +b$ with $a\in \F_q^\times$ and $b\in \F_q$, then we get exactly $q\cdot (q-1)/2$ elements. 
	
	(c) Consider the orbit of $z=(u:1)$ with $u^2=0$. The $G$-orbit  has length $(q^2-1)/2$. The elements with height $0$ 
	are $(\frac{au+b}{cu+d}:1)$, $d\neq 0$; the elements with height $\geq 1$ are $(1:\frac{cu}{au+b})$. These two sets 
	are stable under the action of $B$, so there are at least two orbits. The $B$-orbit of $(u:1)$ consists of the elements 
	$a^2u+b$, $a\in \F_q^\times$, $b\in \F_q$, so its length is $q(q-1)/2$. The $B$-orbit of $(1:u)$ has at least $(q-1)/2$ elements 
	(these are $(1:a^2u)$, $a\in \F_q^\times$). Since $q(q-1)/2+(q-1)/2=(q^2-1)/2$, we see that these two $B$-orbits fill 
	the whole $G$-orbit. 
	
(d) This is clear since the stabilizer  of these points is $\{\pm 1\}$ in $G$. 
\end{proof}

Let $d=r\cdot \deg(\fp)$.  We calculate the $G_{d-1}$-orbits on $\displaystyle{\p(\A/\fp^r)=\bigcup_{0\leq h\leq r} \p(h)}$, where 
$$
\p(h)=\{z\in \p(A/\fp^r)\mid h(z)=h\}. 
$$
From the definition of height, it is easy to see that $h(\gamma z)=h(z)$ for $\gamma\in G_{i}$, $i\geq 1$. Thus, $\p(h)$ 
is stable under $G_{d-1}$. 

\begin{prop}\label{propCusps-h}
 $\p(h)$ splits as follows under $G_{d-1}$: 
	\begin{itemize}
		\item[(a)] $h=0$: One orbit of length $\abs{\fp}^r$, with order of stabilizer of an element equal to $(q-1)$;
		\item[(b)] $1\leq h\leq \lfloor r/2\rfloor$: $2\abs{\fp}^{h-1}\frac{\abs{\fp}-1}{q-1}$ orbits of length $\frac{q-1}{2}\abs{\fp}^{r-2h}$, 
		with stabilizer of order $2\abs{\fp}^{2h}$;
		\item[(c)] $ \lfloor r/2\rfloor< h< r$: $2\abs{\fp}^{r-h-1}\frac{\abs{\fp}-1}{q-1}$ orbits of length $2$, with stabilizer of order $2\abs{\fp}^r$; 
		\item[(d)] $h=r$: One orbit of length $1$, order of stabilizer $(q-1)\abs{\fp}^r$. 
	\end{itemize}
\end{prop}
\begin{proof}  Let $\gamma=\begin{pmatrix} a & b\\ 0 & a^{-1}\end{pmatrix}\in G_{d-1}$. 
	
	If $h=0$, then $z=(u:1)$ 
	and $\gamma z = (a(au+b):1)$. Therefore, $\gamma\in \Stab_{G_{d-1}}(z)$ if and only if $a\in \F_q^\times$ and $b=(1-a^2)u/a$. 
	This implies that $\#\Stab_{G_{d-1}}(z)=q-1$. Hence the length of the orbit of $z$ is 
	$$\#G_{d-1}/\#\Stab_{G_{d-1}}(z) =(q-1)\abs{\fp}^r/(q-1)=\abs{\fp}^r. 
	$$
	Since $\# \p^1(0)=\abs{\fp}^r$, points of height $0$ form a single orbit. 
	
	If $1\leq h\leq r-1$, then $z=(1:w\fp^h)$, where $w\in A$ has degree $\leq \deg(\fp)(r-h)-1$ and is relatively prime to $\fp$. 
	In this case, $\gamma\in \Stab_{G_{d-1}}(z)$ if and only if $(a+wb\fp^h)^{-1}a^{-1}w\fp^h= w\fp^h$ in $A/\fp^r$. 
	This is equivalent to $a^{-1}w\fp^h=(a+wb\fp^h) w\fp^h$, which itself is equivalent to $a=\pm 1$ and $\fp^r\mid b\fp^{2h}$. 
	
	If $1\leq h\leq \lfloor r/2\rfloor$, then $\fp^{r-2h}\mid b$. The number of such elements of degree $\leq d-1$ 
	is $\abs{\fp}^{2h}$. Thus, $\#\Stab_{G_{d-1}}(z)=2\abs{\fp}^{2h}$. The 
	length of the orbit of $z$ is $(q-1)\abs{\fp}^r/2\abs{\fp}^{2h}=\frac{q-1}{2}\abs{\fp}^{r-2h}$.  The number of orbits is 
	$$
	\# \p(h)/ \left(\frac{q-1}{2}\abs{\fp}^{r-2h}\right)= 2\frac{\abs{\fp}^{r-h}-\abs{\fp}^{r-h-1}}{(q-1)\abs{\fp}^{r-2h}} = 2\abs{\fp}^{h-1}\frac{\abs{\fp}-1}{q-1}. 
	$$
	
	If $ \lfloor r/2\rfloor< h< r$, then $b$ can be an arbitrary element of degree $\leq d-1$, so $\#\Stab_{G_{d-1}}(z)=2\fp^r$. The 
	length of the orbit of $z$ is $(q-1)\abs{\fp}^r/2\abs{\fp}^r=\frac{q-1}{2}$.  The number of orbits is 
	$$
	\# \p(h)/ \frac{q-1}{2}= 2\frac{\abs{\fp}^{r-h}-\abs{\fp}^{r-h-1}}{q-1} = 2\abs{\fp}^{r-h-1}\frac{\abs{\fp}-1}{q-1}. 
	$$
	
	If $h=r$, then $z=(0:1)$. In this case, clearly there is only one orbit of length $1$ and $\Stab_{G_{d-1}}(z)=G_{d-1}$. 
\end{proof}

Adding the number of orbits of various height from Proposition \ref{propCusps-h}, we get 
\begin{align}\label{eqNcusps}
\# \text{cusps} = \# X_{d-1} &= 2+ 2\frac{\abs{\fp}-1}{q-1}\left(\sum_{h=1}^{ \lfloor r/2\rfloor} \abs{\fp}^{h-1} + 
\sum_{h=\lfloor r/2\rfloor+1}^{r-1}\abs{\fp}^{r-h-1}\right) \\ 
\nonumber &=  2+ \frac{2}{q-1} \left(\abs{\fp}^{\lfloor r/2\rfloor}+\abs{\fp}^{\lfloor (r-1)/2\rfloor}-2\right). 
\end{align}

\subsection{The general case}  

Now assume $0\neq \fn\in A$ is arbitrary and $\fn=\fp_1^{r_1}\cdots \fp_s^{r_s}$ is its prime decomposition.  
Correspondingly, $\p(A/\fn)$ splits according to the prime decomposition of $\fn$: 
\begin{equation}\label{eqP1decom}
\p^1(A/\fn)\overset{\sim}{\To} \p^1(A/\fp_1^{r_1})\times \cdots\times  \p^(A/\fp_s^{r_s}),  
\end{equation}
with $\SL_2(A)$ acting diagonally.  
Under this decomposition, the stabilizer in $G_n$, $n\geq 0$, of a point $\underline{z}=(z_1, \dots, z_s)\in \p(A/\fn)$ is 
\begin{equation}\label{eqStabG}
\Stab_{G_n}(\underline{z}) =\bigcap_{i=1}^s \Stab_{G_n}(z_i). 
\end{equation}
Let 
$$
\epsilon(\fn)\colonequals \#\p^1(A/\fn) = \prod_{i=1}^s\abs{\fp_i}^{r_i-1}(\abs{\fp_i}+1),
$$
where the second equality follows from \eqref{eqP1decom} and Lemma \ref{lemOrdP-prime}. Also, define
$$
\me(\fn)=
\begin{cases}
	1, & \text{if $\deg(\fp_i)$ is even for all $1\leq i\leq s$;} \\ 
	0, & \text{otherwise.}
\end{cases}
$$

Let 
$$U=\left\{\begin{pmatrix} 1 & b\\ 0 & 1\end{pmatrix} \mid b\in \F_q\right\}, \quad 
\T=\left\{\begin{pmatrix} a& 0\\ 0 & a^{-1}\end{pmatrix} \mid a\in \F_q^\times\right\}, \quad \T^\mathrm{ns}\cong (\F_{q^2}^\times)^1, 
$$ 
where $\T^\mathrm{ns}$ is any of the subgroups of $G$ resulting from an embedding \eqref{eqNonSplTor} (any two of these are 
conjugate in $\GL_2(\F_q)$). Given a subgroup $H$ of $G$ and $g\in \GL_2(\F_q)$, 
let $H^g\colonequals g^{-1}H g\subseteq G$. The following lemma is an immediate consequence of \cite[Lem. 2.6]{GN}: 

\begin{lem}\label{lem:Lem2.6GN} Let $g, h, k\in \GL_2(\F_q)$. 
	\begin{enumerate}
		\item $\# (B^g\cap U^h)>1 \Rightarrow U^h\subset B^g$; 
		\item $B^g\neq B^h \Rightarrow \exists k\in \GL_2(\F_q) \text{ such that }B^g\cap B^h\cong \T$; 
		\item $B^g, B^h, B^k$ pairwise different $\Rightarrow B^g\cap B^h\cap B^k=\{\pm 1\}$; 
		\item $B^g\cap \T^\mathrm{ns}=\{\pm 1\}$; 
		\item $(\T^\mathrm{ns})^g\neq (\T^\mathrm{ns})^h \Rightarrow (\T^\mathrm{ns})^g\cap (\T^\mathrm{ns})^h=\{\pm 1\}$.  
		\end{enumerate}
\end{lem}

\begin{prop}\label{prop:OrbitsofG-genN} The points of $\p^1(A/\fn)$, with their $G$-stabilizers and orbits, are:
	\begin{itemize}
		\item[$($a$)$] 
		$\displaystyle{\begin{cases}
				\underline{z}\in \p^1(\F_q)\overset{\mathrm{diag}}{\longhookrightarrow} \p^1(A/\fp^r), \text{i.e., $z_i=u$ for all $i$ and some fixed $u\in \p^1(\F_q)$}. \\
				\Stab_G(\underline{z}) \cong B. \\
				\text{There is one orbit of this type, of length } (q+1). 
		\end{cases}}$
	
	\item[$($a*$)$] 
	$\displaystyle{\begin{cases}
			\text{There exists a non-trivial disjoint partition $S\cup S'$ of $\{1,2,\dots, s\}$}\\ 
			\text{and $u\neq v\in \p^1(\F_q)$ such that $z_i=u$ if $i\in S$ and $z_i=v$ if $i\in S'$}. \\ 
			\Stab_G(\underline{z}) \cong \left\{\begin{pmatrix} a & 0 \\ 0 & a^{-1}\end{pmatrix}\mid a\in \F_q^\times\right\}\cong \F_q^\times. \\
			\text{There are $2^{s-1}-1$ orbits of this type, each of length } q(q+1). 
	\end{cases}}$
		
		\item[$($b$)$] 
		$\displaystyle{\begin{cases}
				\text{For all $1\leq i\leq s$, $\deg(\fp_i)$ is even and there an irreducible quadratic}\\  
				\text{$x^2+ax+b$ over $\F_q$ such that each $z_i$ is a root of this quadratic.} \\ 
				\Stab_G(\underline{z})\cong (\F_{q^2}^\times)^1. \\ 
				\text{There are $2^{s-1}$ orbits of this type, each of length }q(q-1). 
		\end{cases}}$
		
		\item[$($c$)$] 
		$\displaystyle{\begin{cases}
				\underline{z}\not\in \p^1(\F_q)\overset{\mathrm{diag}}{\longhookrightarrow} \p^1(A/\fp^r), \text{but there exists $y\in \p^1(\F_q)$ such that for all $1\leq i\leq s$: }\\ 
			\quad 	\text{if $y=(c:1)$ then $z_i=(u_i:1)$, $(u_i-c)^2=0$, }\\ 
			\quad 	\text{if $y=(1:0)$ then $z_i=(1:v)$, $v^2=0$}. \\ 
				\Stab_G(\underline{z})\cong \left\{\pm  \begin{pmatrix} 1 & b\\ 0 & 1\end{pmatrix} \mid b\in \F_q\right\}; \\  
				\text{There are $\frac{2}{q-1}\prod_{i=1}^s (\abs{\fp_i}^{\lfloor r_i/2\rfloor}-1)$ orbits of this type, each of length $(q^2-1)/2$. }
		\end{cases}}$
		
		\item[$($d$)$] 
		$\displaystyle{\begin{cases}
				\text{All other points}.\\ 
				\Stab_G(\underline{z})\cong \{\pm 1\}.\\
				\text{There are} \\ \frac{2}{q(q^2-1)}\left(\epsilon(\fn)-(q+1)-(2^{s-1}-1)q(q+1) -2^{s-1} q(q-1)\me(\fn)-(q+1)\prod_{i=1}^s (\abs{\fp_i}^{\lfloor r_i/2\rfloor}-1)\right)\\
			\text{orbits of this type, each of length }q(q^2-1)/2. 
		\end{cases}}$
		
	\end{itemize}
\end{prop}
\begin{proof}
	Suppose $\underline{z}=(z_1, \dots, z_s)$ has a stabilizer in $G$ strictly larger than the center $Z\colonequals \{\pm 1\}$. Due to \eqref{eqStabG}, 
	each $\Stab_{G_0}(z_i)$ is strictly larger than $Z$, so each $z_i$ must be of one of the types (a)-(c) in Proposition \ref{prop:OrbitsofG1}. 
	Moreover, because of the intersection properties in Lemma \ref{lem:Lem2.6GN}, all $z_i$ must be of type (a) or (c), or all must be of type (b). 
	(It is well-known 
	that the stabilizers of the same type are conjugate in $\GL_2(\F_q)$, e.g., any two Borel subgroups or non-split tori are conjugate, so we are 
	in a position to apply Lemma \ref{lem:Lem2.6GN}.) 
	
	If the $z_i$'s are all of type (a) but three of them are pairwise distinct, then $\Stab_{G}(\underline{z})=Z$ by  Lemma \ref{lem:Lem2.6GN} (3), 
	contradicting our assumption. Therefore, either all $z_i$'s are equal (case (a)), or they fall into two distinct subsets, with the elements 
	in each subset being equal to the same point of $\p^1(\F_q)$ (case (a*)). The number of points of type (a*) is 
	$$
	{q+1 \choose 2}\left({s \choose 1}+{s \choose 2}+\cdots+ {s \choose s-1}\right) =q(q+1)(2^{s-1}-1),
	$$
	where ${q+1 \choose 2}$ is the number of choices of two distinct points of $\p^1(\F_q)$ and $\sum_{i=1}^{s-1}{s \choose i}$ is the 
	number of ways of subdividing a set with $s$ elements into two disjoint non-empty subsets.  The stabilizer of each point of 
	type (a*) is isomorphic to $\F_q^\times$ by Lemma \ref{lem:Lem2.6GN} (2), so the $G$-orbit of such point has length $\#G/(q-1)=q(q+1)$. 
	Therefore, the number of orbits is $2^{s-1}-1$. This proves (a) and (a*). 
	
	If the $z_i$'s are of type (b), then they satisfy an irreducible quadratic equation; this equation must be the \textit{same} for all $i$ by 
	Lemma \ref{lem:Lem2.6GN} (5). Since by Proposition \ref{prop:OrbitsofG1} there are precisely $2$ solutions in each $A/\fp_i^{r_i}$, there 
	are $\frac{q(q-1)}{2}\cdot 2^s$ such $\underline{z}$, where $\frac{q(q-1)}{2}$ is the number of Galois conjugate pairs of points in $\p^1(\F_{q^2})-\p^1(\F_q)$.  
	The stabilizer of $\underline{z}$ is isomorphic to $(\F_{q^2}^\times)^1$, so its orbit has length $q(q^2-1)/(q+1)=q(q-1)$. Thus, the 
	number of orbits of elements of type (b) is $2^{s-1}$. This proves (b). 
	
	Now suppose the $z_i$'s are either of type (a) or (c) of Proposition \ref{prop:OrbitsofG1}, with at least one $z_i$ being of type (c). 
		Because of Lemma \ref{lem:Lem2.6GN} (1) and (2), the distinguished element $e\in \p^1(\F_q)$ must be the same for all $i$:  
		if $e=(c:1)$ then  $z_i=(u_i:v_i)$ and $(u_i-c)^2=0$ for all $i$; if $e=(1:0)$, then $z_i=(u_i:v_i)$  and $v_i^2=0$ for all $i$. 
		The number of such elements is $$(q+1)\prod_{i=1}^s(\abs{\fp_i}^{\lfloor r_i/2\rfloor}-1), $$ where $(q+1)$ corresponds to 
		the choice of $e\in \p^1(\F_q)$ and the factors $(\abs{\fp_i}^{\lfloor r_i/2\rfloor}-1)$ result from the calculations leading to 
		Proposition \ref{prop:OrbitsofG1} (they are the number of  solutions of $(u_i-c)^2=0$ and $v_i^2=0$ in $A/\fp_i^{r_i}$). The 
		stabilizer of each element has order $2q$, so the length of the orbit of such element is $\# G/2q=(q^2-1)/2$. Thus the number 
		of orbits of this type is $\frac{2}{q-1}\prod_{i=1}^s (\abs{\fp_i}^{\lfloor r_i/2\rfloor}-1)$. This proves (c). 
\end{proof}

\begin{prop}\label{propOrbitsofB-genN} The $G$-orbits of Proposition \ref{prop:OrbitsofG-genN} split as follows into $B$-orbits: 
	\begin{itemize}
		\item[(a)] Two $B$-orbits. 
		\item[(a*)] Four $B$-orbits. 
		\item[(b)] Two $B$-orbits. 
		\item[(c)] Two $B$-orbits. 
		\item[(d)] $(q+1)$ $B$-orbits. 
	\end{itemize}
\end{prop}
\begin{proof}
	(a) The $G$-orbit $\{\underline{z}\mid z_1=\cdots=z_s\in \p^1(\F_q)\}$ under the diagonal action of $B$ splits into 
	$\{\underline{z}\mid z_1=\cdots=z_s=(u:1)\in \p^1(\F_q)\}$  and $\{\underline{z}\mid z_1=\cdots=z_s=(1:0)\}$.  
	
	(a*) The $G$-orbit associated with the partition $S\cup S'=\{1, \dots, s\}$ splits into
	\begin{itemize}
		\item one $B$-orbit $\{\underline{z}\mid z_i=(1:0) \text{ for } i \in S, z_i\neq (1:0) \text{ for } i \in S'\}$ of length $q$, 
		\item one $B$-orbit $\{\underline{z}\mid z_i=(1:0) \text{ for } i \in S', z_i\neq (1:0) \text{ for } i \in S\}$ of length $q$, 
		\item two $B$-orbits $\{\underline{z}\mid z_i\neq (1:0) \text{ for all }i\in S\cup S'\}$, each of length $q(q-1)/2$. 
	\end{itemize}
	The first two cases are clear. For the third case note that for $\underline{z}$ for which no $z_i$ 
	is equal to $(1:0)$, we have  $\Stab_B(\underline{z})=B\cap \Stab_G((u:1))\cap \Stab_G((v:1))$, $u\neq v$. 
	This is an intersection of three distinct Borel subgroups of $G$, so by Lemma \ref{lem:Lem2.6GN} (3) it is equal to $Z=\{\pm 1\}$. Thus, the length 
	of the $B$-orbit of $\underline{z}$ is $q(q-1)/2$. Since the length of the $G$-orbit is $q(q+1)$ and $q+q+2\frac{q(q-1)}{2}=q(q+1)$, 
	we see that there are two $B$-orbits of this last type. 
	
	(b) This follows from the same argument as the proof of Proposition \ref{prop:orbitsB} (b). 
	
	(c) Similar to the proof of Proposition \ref{prop:orbitsB} (c), the $G$-orbit has length $(q^2-1)/2$ and decomposes into 
	two $B$-orbits given by $\{\underline{z}\mid h(z_i)\geq 1 \text{ for all }i\}$ of length $(q-1)/2$ and 
	$\{\underline{z}\mid h(z_i)=0 \text{ for all }i\}$ of length $q(q-1)/2$ .
\end{proof}

Let $\underline{h}=(h_1, \dots, h_s)$, $0\leq h_i\leq r_i$ for $i=1,2, \dots, s$, and 
$$
\p^1(\underline{h})=\{\underline{z}\in \p^1(A/\fn)\mid h(z_i)=h_i\}. 
$$
Note that $\p^1(\underline{h})$ is stable under the action of $G_{d-1}$, $d=\deg(\fn)$. 

\begin{prop}\label{propNumCuspsN} $\p^1(\underline{h})$ splits as follows under $G_{d-1}$:
	\begin{itemize}
		\item[(a)] Each $h_i=0$ or $r_i$; one orbit; 
		\item[(b)] There exists $i$ such that $0<h_i<r_i$; the number of orbits is 
		$$
		\frac{2}{q-1} \prod_{\substack{1\leq i\leq s\\ 0<h_i<r_i}}\abs{\fp_i}^{m_i-1}(\abs{\fp_i}-1)\quad \text{with }m_i=\min\{h_i, r_i-h_i\}. 
		$$
	\end{itemize}
\end{prop}
\begin{proof} With respect to the factorization $\p^1(A/\fn)=\prod_{i=1}^s\p^1(A/\fp_i^{r_i})$,  the group $G_{d-1}$ acts on $\p^1(A/\fn)$ as  
$$\F_q^\times \prod_{i=1}^s U(A/\fp_i^{r_i}), 
	$$ 
	where $\F_q^\times\cong \left\{\begin{pmatrix} a & 0 \\ 0 & a^{-1}\end{pmatrix}\mid a\in \F_q^\times\right\}$ acts 
	diagonally and $$U_i\colonequals U(A/\fp_i^{r_i}) =  \left\{\begin{pmatrix} 1 & b \\ 0 & 1\end{pmatrix}\mid b\in A/\fp_i^{r_i}\right\}. $$
	
	The group $U_i\colonequals U(A/\fp_i^{r_i})$ 
	fixes the point $(1:0)\in \p^1(A/\fp_i^{r_i})$, which has height $r_i$, and permutes transitively the points $(u:1)$, $u\in A/\fp_i^{r_i}$, which 
	have height $0$.  This implies (a). 
	
	Next, if $0< h_i<r_i$, from the proof of Proposition \ref{propCusps-h} it follows that $\# \Stab_{U_i}(z_i) =\abs{\fp}^{\min\{2h_i, r_i\}}$. Thus, 
	$$
	\#\Stab_{G_{d-1}}(\underline{z}) = 2 \prod_{i=1}^s \abs{\fp}^{\min\{2h_i, r_i\}},
	$$
	where $2$ comes from $\{\pm 1\}\in \F_q^\times$. Hence the orbit of $\underline{z}$ has length 
	$$
	\#G_d/\# \Stab_{G_{d-1}}(\underline{z}) =\frac{q-1}{2}\prod_{i=1}^s \abs{\fp}^{r_i-\min\{2h_i, r_i\}}
	$$
On the other hand, by \eqref{eqOrdPh},  
$$
\# \p^1(\underline{h}) = \prod_{\substack{1\leq i\leq s\\ h_i=0 \text{ or }r_i}}\abs{\fp_i}^{r_i-h_i} 
\prod_{\substack{1\leq i\leq s\\ 0<h_i < r_i}}(\abs{\fp_i}-1)\abs{\fp_i}^{r_i-h_i-1}. 
$$
Dividing $\# \p^1(\underline{h})$ by  $	\#G_d/\# \Stab_{G_{d-1}}(\underline{z})$ we get the claimed formula for the number of orbits. 
\end{proof}

\begin{defn}
	The cusps corresponding to $G_{d-1}$-orbits of type (a) above are called \textit{regular}; there are $2^s$ of them. 
	The cusps of type (b) are called irregular. 
\end{defn}

Comparing the formulas in Proposition \ref{propNumCuspsN} with the corresponding formulas in \cite[Prop. 2.14]{GN} for $\G_0(\fn)\bs \cT$ 
(note that there is a typo in \textit{loc. cit.}), 
one concludes that  in the graph covering $$\G_0^1(\fp^r)\bs \cT \To \G_0(\fp^r)\bs \cT$$ all regular cusps ramify and all irregular cusps do not ramify.  
Here ``ramifies" (resp. ``does not ramify") means that the preimage of a given half-line in $\G_0(\fp^r)\bs \cT$ 
consists of one (resp. two) half-lines in $\G_0^1(\fp^r)\bs \cT$ (see the appendix for some explicit examples). 

\begin{cor}\label{corNumCusps-GenN}
	The number of cusps of $\G_0^1(\fp^r)\bs \cT$ is 
	$$
	2^s+ \frac{2}{q-1}(\kappa(\fn)-2^s), 
	$$
	where 
	$$
	\kappa(\fn)=\prod_{1\leq i\leq s} \left(\abs{\fp_i}^{\lfloor (r_i-1)/2\rfloor}+ \abs{\fp_i}^{\lfloor r_i/2\rfloor}\right). 
	$$
\end{cor}
\begin{proof} One needs to add the numbers of cusps of various heights $\underline{h}$. With the help of 
	Proposition \ref{propNumCuspsN} this reduces to the same calculation as in the proof of Lemma 2.16 in \cite{GN}. 
\end{proof}

Combining Proposition \ref{prop:OrbitsofG-genN}, Proposition \ref{propOrbitsofB-genN}, Corollary \ref{corNumCusps-GenN}  and 
\eqref{eqGenus}, we get 
$$
\boxed{
g(\G_0^1(\fn)\bs \cT ) = 1 +\frac{2}{q^2-1} \epsilon(\fn) - \frac{2}{q-1}\kappa(\fn)+2^{s-1}\left( \frac{3-q}{q-1}+\frac{1-q}{q+1}\me(\fn)\right)}
$$
Comparing the above formula with the formula for $g(\G_0^1(\fn)\bs \cT )$ given in \cite[Thm. 2.17]{GN}, gives the simple relationship 
\begin{equation}\label{eqGenRH}
g\left(\G_0^1(\fn)\bs \cT\right) =2 \cdot g\left(\G_0(\fn)\bs \cT\right)-1+ 2^{s-1}(1+\me(\fn)). 
\end{equation}

\begin{rem}
	Comparing \eqref{eqGenRH} with the Riemann-Hurwitz formula, we recognize the term $2^s(1+\me(\fn))$ as the ramification 
	of the covering $X_0^1(\fn)\to X_0(\fn)$. We know that the regular cusps contribute $2^s$ to the ramification, and the other cusps do not ramify. 
	Hence $2^s$ elliptic points of $X_0(\fn)$ ramify in $X_0^1(\fn)\to X_0(\fn)$ if $\me(\fn)=1$, and otherwise only the regular cusps ramify. 
	The ramified elliptic points of $X_0(\fn)$ are the ones that are unramified over $X_0(1)$. 
\end{rem}

\begin{cor}\label{corg=1} We have 
	\begin{itemize}
	\item $g\left(\G_0^1(\fn)\bs \cT\right)=0$ if and only if either $\fn$ is a prime of degree $1$ or $\fn$ is a square of such prime. 
	  \item $g\left(\G_0^1(\fn)\bs \cT\right)=1$ if and only if either $\fn$ is a product of two distinct primes of degree $1$, or $\fn$ is a prime of degree $2$. 
	\end{itemize}
\end{cor}

See the appendix for pictures of the quotient graphs in the corollary.   


\section{$X_0^1(\fn)$ of genus $1$: equations}\label{sEquations}
In this section, for the cases when $X_0^1(\fn)$ has genus $1$ (see Corollary \ref{corg=1}), we deduce 
the defining  Weierstrass equations. The idea of our approach is fairly simple: 
\begin{enumerate}
	\item In the two indicated cases, $X_0(\fn)$ has genus $0$. First, we find its ``Hauptmodul", i.e., a generator $z$ of the function field $\cF_{0, \fn}$ of $X_0(\fn)$. 
	\item We express the $j$-function as a rational function of $z$, $j=a(z)/b(z)$, where $a(z), b(z)\in \Ci[z]$. 
	\item Finally, by Lemma \ref{lemFFX0}, the function field $\cF_{1, \fn}$ of $X_0^1(\fn)$ is obtained from $\cF_{0, \fn}$ by adjoining the square root of $j$. 
	If we identity $\cF_{0, \fn}=\Ci(z)$,  then 
	$$\cF_{1, \fn}=\Ci(z)[\sqrt{a(z)/b(z)}] = \Ci(x)[\sqrt{a(z)\cdot b(z)}]. 
	$$ 
	Because $X_0^1(\fn)$ has genus $1$, the polynomial $a(z)b(z)$, up to a square, is a monic square-free polynomial $f(z)$ of degree $3$. The Weierstrass 
	equation of $X_0^1(\fn)$ is $y^2=f(z)$. 
\end{enumerate} 

Note that in the case when $\fn$ is a product of two distinct primes of degree $1$, by applying an automorphism of $F$, we may assume that $\fn=T(T+1)$. 
We treat the case when $\fn=T(T+1)$ and $\fn$ is irreducible of degree $2$ separately since the actual calculations in steps (1) and (2) above 
are quite different. 


\subsection{Equation of $X_0^1(T(T+1))$}\label{ssX01T(T+1)}
Let 
\begin{equation}\label{eqUDM}
\phi_T(x)=Tx+x^q+j^{-1}x^{q^2},
\end{equation}
be the ``universal" Drinfeld module of rank $2$ with $j$-invariant $j$, where we consider $j$ as a variable.  
A cyclic $T$-submodule $\cC$ of $\phi$ is the set of roots of an $\F_q$-linear polynomial of the form $f_\cC(x)=x+\alpha x^q$, where 
$\alpha$ is another variable 
(we may assume that the coefficient of $x$ in $f_\cC(x)$ is $1$ because the polynomial is separable and the set of zeros does not change 
if we multiply $f_\cC$ by a nonzero constant). Since $\cC\subset \phi[T]$, we have 
$$
\phi_T(x)= (Tx+\tilde{\alpha}x^q)\circ (x+\alpha x^q). 
$$
Thus, $j^{-1}=\tilde{\alpha}\alpha^q$ and $\alpha T+\tilde{\alpha}=1$. This leads to $j^{-1}=(1-\alpha T)\alpha^q$. Substituting $\alpha\mapsto -1/\alpha$, we obtain 
$$
T+\alpha+j^{-1}\alpha^{q+1} = 0. 
$$
Note that $\cC$ is automatically a $\phi(A)$-module, i.e., $f_\cC\circ \phi_T=g\circ f_\cC$ for some $g\in \atwist{\Ci}{x}$, since any root of $f_\cC(x)$ 
maps to $0$ under the action of $\phi_T$. Thus, $X_0(T)$ is defined by 
$$
j=-\frac{\alpha^{q+1}}{\alpha+T}. 
$$ 
Similarly, $X_0(T+1)$ is defined by 
$j=-\beta^{q+1}/(\beta+(T+1))$ for another variable $\beta$. Since a cyclic $T(T+1)$-submodule of $\phi_T$ decomposes into a direct product of cyclic $T$ and $T+1$ submodules, we 
conclude that $\cF_{0, T(T+1)}=\Ci(j, \alpha, \beta)=\Ci(\alpha, \beta)$. Changing the notation for the  variables  $\alpha$ and $\beta$ to the 
more conventional $x$ and $y$, we obtain the following as an equation of the curve $X_0(T(T+1))$ in the affine plane $\Spec(\Ci[x, y])$:
\begin{equation}\label{eqofX_0(T(T+1))}
	\frac{x^{q+1}}{T+x} = 	\frac{y^{q+1}}{(T+1)+y}.  
\end{equation}

Let $z$ be a generator of the function field of $\p^1_{\Ci}$. Put $S\colonequals 1/(T^q+1)=1/(T+1)^q$. Then, as one checks by a straightforward calculation, 
$$
z\longmapsto (x(z), y(z))=\left(\frac{-T(z^{q+1}+z^q+S^{q-1}z+S^q)}{z^{q+1}+Sz^q+S^{q-1}z+S^q}, 
\frac{-(T+1)(z^{q+1}+z^q+S^{q-1}z+S^q)}{z^{q+1}+S^{q-1}z}\right), 
$$
defines a morphism $\p^1_{\Ci}\to X_0(T(T+1))$, i.e., $x(z)$ and $y(z)$ satisfy \eqref{eqofX_0(T(T+1))}. Since this morphism 
has degree $1$, it is an isomorphism. Write 
$$
j(z)= -\frac{x(z)^{q+1}}{T+x(z)} =\frac{a(z)}{b(z)},
$$
where $a(z), b(z)\in \Ci[z]$. One checks by a tedious calculation that 
\begin{align}\label{eq-azbz}
a(z)b(z) = &(T^q+1) z^q\left(z+ \frac{1}{(T+1)^q}\right)^q \left(z+ \frac{1}{(T+1)^{q-1}}\right)^{q^2} \\ 
\nonumber & \times \left(z^{q+1}+z^q+S^{q-1}z+S^q\right)^{q+1}. 
\end{align}
As was explained at the beginning of this section, extracting the square-free part from the above polynomial, 
we obtain a Weierstrass equation for $X_0^1(T(T+1))$: 
\begin{equation}\label{eqX(T(T+1))}
y^2=x\left(x+ \frac{1}{(T+1)^q}\right) \left(x+ \frac{1}{(T+1)^{q-1}}\right).
\end{equation}
(Here we again changed the notation for the variable to $z$ to the more conventional $x$.)
	
	The elliptic curve given by \eqref{eqX(T(T+1))} is visibly defined over $F$. Thus, the modular curve $X_0^1(T(T+1))$ 
	has a model over $F$, but the model \eqref{eqX(T(T+1))} is not ``canonical" in a modular sense.  
	In fact, throughout this subsection we could have worked over $F$, since  \eqref{eqofX_0(T(T+1))} and the explicit parametrization $\p^1_F\to X_0(T(T+1))$ 
	are defined over $F$. Extracting the square-free part of \eqref{eq-azbz} over $F$, we end up with 
	$y^2=(T+1)x(x+1/(T+1)^q)(x+1/(T+1)^{q-1})$. So, substituting $y\mapsto y/(T+1)$ and $x\mapsto x/(T+1)$, we get the following Weierstrass 
	equation for $X_0^1(T(T+1))$: 
	\begin{equation}\label{eqX(T(T+1))-2}
		\boxed{y^2=x\left(x+ \frac{1}{(T+1)^{q-1}}\right) \left(x+ \frac{1}{(T+1)^{q-2}}\right)}
\end{equation}
	There are some differences between the models \eqref{eqX(T(T+1))} and \eqref{eqX(T(T+1))-2} of $X_0^1(T(T+1))$. The elliptic 
	curve $E_1$ defined by  \eqref{eqX(T(T+1))}  has conductor $\fn_{E_1}=T\cdot (T+1)^2\cdot \infty$, 
	and the elliptic curve  $E_2$ defined by \eqref{eqX(T(T+1))-2} has conductor $\fn_{E_2}=T\cdot (T+1)\cdot \infty^2$. In particular, these 
	curves are not isomorphic over $F$. 
	It is known that the conductor of a non-isotrivial elliptic curve over $F$ must have degree $\geq 4$, so both $\fn_{E_i}$ are minimal in that sense. 
	Elliptic curves over $F$ with conductor of degree $4$ are called \textit{extremal}; such curves have special geometric properties; cf. \cite{Ito}. 

\subsection{Equation of $X_0^1(\fp)$} In this subsection, we assume that $\fp=T^2+aT+b\in \F_q[T]$ is irreducible of degree $2$.  
One can pursue the same strategy as in $\S$\ref{ssX01T(T+1)} to obtain an equation for $X_0(\fp)$, but this results in a system of 
polynomial equations rather than a convenient plane model. Instead we use modular forms to construct an explicit generator of $\cF_{0, \fp}$; 
this is partly motivated by \cite{YangAM}.  

Let $\Delta(z)$  be the Drinfeld discriminant function defined in $\S$\ref{ssDMF} and let $\Delta_\fp(z)\colonequals \Delta(\fp z)$. 
The quotient $\Delta/\Delta_\fp$ is invariant under the action of $\G_0(\fp)$. It was shown by Gekeler (see \cite[(3.18), (3.20), (3.22)]{GekelerCM97}) that there is a 
modular form $\eta(z)$ for $\G_0(\fp)$ of weight $0$ and type $0$ such that $\eta^{q^2-1}=\Delta/\Delta_\fp$. Moreover, from the proof of Corollary 3.23 in \cite{GekelerCM97}, 
the divisor of $\eta$ on $X_0(\fp)$ is $[0]-[\infty]$, where $[0]$ and $[\infty]$ are the cusps of $X_0(\fp)$. Thus, $\Ci(\eta)=\cF_{0, \fp}$. The 
problem now becomes expressing $j(z)$ as a rational function in $\eta(z)$.  

\begin{rem}
	For general irreducible $\fp$, the cuspidal divisor group $\cC(\fp)$ of $X_0(\fp)$ is cyclic of order $(\abs{\fp}-1)/(q^2-1)$ if $\deg(\fp)$ is even 
	and $(\abs{\fp}-1)/(q-1)$ if $\deg(\fp)$ is odd; see \cite[Cor. 3.23]{GekelerCM97}. The cuspidal divisor group of $X_0(\fp)$ is the subgroup 
	of its Jacobian generated by the differences of cusps. Since there are only two cusps in this case, the cuspidal divisor group is clearly cyclic. 
	The calculation of its order is closely related to the calculation of the largest root of $\Delta/\Delta_\fp$ in the group of modular units on $\Omega$. 
	Gekeler's result can be extended to $X_0^1(\fp)$ as follows: The cuspidal divisor group $\cC^1(\fp)$ of $X_0^1(\fp)$ is cyclic of order 
	$2 (\abs{\fp}-1)/(q^2-1)$ if $\deg(\fp)$ is even 
	and $(\abs{\fp}-1)/(q-1)$ if $\deg(\fp)$ is odd. The proof is similar to \textit{loc. cit.} 
	Again, $X_0^1(\fp)$ has only two cusps, which, by abuse of notation, 
	we denote $[0]$ and $[\infty]$, depending on which cusp of $X_0(\fp)$ each maps to under $X_0^1(\fp)\to X_0(\fp)$. 
	The divisor of $\Delta/\Delta_\fp$ on $X_0^1(\fp)$ is $2(\abs{\fp}-1)\cdot([0]-[\infty])$. It follows from (3.21) and (3.22) in \cite{GekelerCM97} 
	that the largest $n$ such that there is a holomorphic nonvanishing function $D_\fp$ on $\Omega$ such that $D_\fp^n=\Delta/\Delta_\fp$ 
	and $D_\fp$ is invariant under the action of $\G_0^1(\fp)$ is $q^2-1$ if $\deg(\fp)$ is even and $2(q-1)$ if $\deg(\fp)$ is odd. 
	This implies that the order of $\cC^1(\fp)$ is given by the formula above. 
\end{rem}

As a function on $X(1)$, $j(z)$ has a simple pole at the unique cusp $[\infty']$ of $X(1)$. In the natural covering $X_0(\fp)\to X(1)$, which has degree $q^2+1$, 
the cusp $[\infty]$ does not ramify over $[\infty']$ but the cusp $[0]$ ramifies with index $q^2$. Thus, as a function on $X_0(\fp)$, $j(z)$
has a pole of order $q^2$ at $[0]$, a simple pole at $[\infty]$, and no other poles. Since $\eta$ has a simple zero at $[0]$, and no other zeros, there 
is a polynomial $f_1(x)\in \Ci[x]$ of degree $\leq q^2$ such that $j - f_1(\eta)/\eta^{q^2}$ has a pole only at $[\infty]$ and a zero at $[0]$. Moreover, because 
$\eta$ has a simple pole at $[\infty]$, the pole of $j - f_1(\eta)/\eta^{q^2}$  at $[\infty]$ is the pole of $j$, so it is simple. Since $\eta$ 
also has a simple pole at $[\infty]$, there is a nonzero constant $c\in \Ci$ such that $j-f_1(\eta)/\eta^{q^2}-c\eta$ has no poles, so it is a constant. 
Because $\eta$ and  $j-f_1(\eta)/\eta^{q^2}$ have zeros at $[0]$, the constant above is zero. We conclude that there is a polynomial $f(x)\in \Ci[x]$ 
of degree $q^2+1$ such that  
$$
j=\frac{f(\eta)}{\eta^{q^2}}. 
$$ 
Once we find $f(x)$, this can serve as a defining equation of $X_0(\fp)$ if $j$ and $\eta$ are treated as indeterminates. Moreover, 
the square-free part of $f(x)$ must be a quadratic polynomial $f_2(x)$ with distinct roots and the equation of $X_0^1(\fp)$ will be $y^2=x\cdot f_2(x)$ 
by the argument outlined at the beginning of this section. To find $f$ we use the $t$-expansions of $j(z)$ and $\eta(z)$. 

Let $\pi_C$ be the Carlitz period, i.e., a generator of the lattice corresponding to the Carlitz module defined by $\rho_T=Tx+x^q$. Let  
$$
t(z)=\frac{1}{\pi_Ce_{A}(z)} = \pi_C^{-1}\sum_{a\in A} \frac{1}{z+a}.
$$ 
This function is the parameter at infinity of $\Omega$ analogous to the classical $\exp(2\pi i z)$. Define the \textit{$a$-th inverse cyclotomic polynomial} 
$$
\theta_a(x)=\rho_a(x^{-1}) x^{\abs{a}}. 
$$
Gekeler proved that (cf. \cite[Thm. 6.1]{Gekeler-Coeff})
\begin{equation}\label{eqDeltaProdExp}
\pi_C^{1-q^2} \Delta(z)= -t^{q-1}\prod_{a\in A \text{ monic}} \theta_a(t)^{(q^2-1)(q-1)} = \sum_{i\geq 1} \delta_i t^{(q-1)i},
\end{equation}
where the product converges for $\Im(z)$ sufficiently large. (Since $\GL_2(A)$ contains 
matrices of the form $\begin{pmatrix} \alpha & 0 \\ 0 & 1\end{pmatrix}$, $\alpha\in \F_q^\times$, modular transformation rule for $\Delta(z)$ 
implies that $\Delta(\alpha z)=\Delta(z)$, so the coefficient of $t^i$ in the $t$-expansion of $\Delta$ is zero if $(q-1)\nmid i$.)
Using \eqref{eqDeltaProdExp}, it is possible to compute the coefficients $\delta_i$, $i\leq N$, for any desired finite $N$. Already in \cite[p. 692]{Gekeler-Coeff}, 
one finds (nonzero) $\delta_i$ for $i\leq 2q^2$: 
$$
\delta_1=-1, \quad \delta_q=1, \quad \delta_{q+1}=-(T^q-T), \quad \delta_{q^2-q+1} = -1, \quad \dots 
$$
Next, by \cite[(6.2)]{Gekeler-Coeff}, 
\begin{align*}
t_\fp(z)\colonequals t(\fp z) &=\frac{t^{\abs{\fp}}}{\theta_\fp(t)}=\frac{t^{q^2}}{\fp t^{q^2-1}+(T^q+T+a)t^{q^2-q}+1} \\ 
& = t^{q^2}-(T^q+T+a) t^{2q^2-q}+\cdots 
\end{align*}
Substituting $t_\fp$ into $\Delta(t)$, one obtains the $t$-expansion of 
$$
\pi_C^{1-q^2}\Delta_\fp(z)=\pi_C^{1-q^2} \Delta(t_\fp) = -t^{(q-1)q^2}+\cdots 
$$
Thus, we can compute (using a computer) the $t$-expansion of 
$
\Delta/\Delta_\fp = t^{-(q^2-1)(q-1)}+\cdots  
$
and also find a $(q^2-1)$th root $\eta(s) := \sqrt[q^2-1]{\Delta/\Delta_\fp} := s^{-1}+\cdots \in \frac{1}{s}\cdot\mathbb{F}_q(T)[[s]]$ with leading coefficient $1$, where $s=t^{q-1}$ (cf. \cite[p. 683]{Gekeler-Coeff}). Next, 
\cite[Cor. 10.11]{Gekeler-Coeff} gives the $s$-expansion of $g(z)$ defined in \eqref{eqCoeffForms}:
$$
\pi_C^{1-q} g(s)=1-(T^q-T)\left(s+s^{q^2-q+1}(1-s^{q-1}+(T^q-T)s^q)^{(1-q)}\right)+\cdots 
$$
and thus also of $j(s)=g^{q+1}/\Delta=-s^{-1}+\cdots$
Substituting the $s$-expansions of $\eta$ and $j$ into $j\eta^{q^2}=f(\eta)$ and treating the coefficients of $f(x)$ as variables, one obtains 
a system of linear equations in those variables that can be uniquely solved (the number of coefficients in the $s$-expansions of $\eta$ and $j$ 
that need to be computed to solve this system depends on $q$). 

\begin{example} The computations in this and the other examples in this section were performed using the computer program \textsc{Magma}. 

	Let $q=3$ and $\fp=T^2+T+2$. We have 
	\begin{align*}
	\eta(s) &=s^{-1}+s+(2T+1)s^2+T(T+1)s^3-(T-1)^3s^4+\cdots  \\
	j(s) & = -s^{-1}+(T^3-T)-s+(T^9+T^3+T)s^2+\cdots  
	\end{align*}
	In this case, we find that 
	$$
	f(x) = -(x^2+(2T+1)\fp x-\fp^2)^4(x^2+(2T+1)x+\fp). 
	$$
	Thus, $X_0^1(\fp)$ can be defined by the equation 
	$$
	y^2=x(x^2+(2T+1)x+\fp). 
	$$
	For the other two irreducibles of degree $2$ in $\F_3[x]$ we get: 
	\begin{align*}
		&\fp=T^2+2T+2 &\leadsto & &f(x)=-(x^2+(2T+2)\fp x-\fp^2)^4\cdot (x^2+(2T+2)x+\fp)\\ 
		&\fp=T^2+1 &\leadsto & &f(x)=-(x^2+(2T)\fp x-\fp^2)^4\cdot (x^2+(2T)x+\fp). 
	\end{align*}
	Thus, writing $\fp=T^2+aT+b$ we conclude that for $q=3$ the equations of modular curves are
	\begin{align*}
	&X_0(\fp)\colon \quad j \cdot x^{q^2} = -(x^2+(2T+a)\fp x-\fp^2)^{(q+1)}(x^2+(2T+a)x+\fp) \\ 
	&X_0^1(\fp)\colon \quad y^2=x(x^2+(2T+a)x+\fp). 
	\end{align*}
\end{example} 

\begin{rem}
	The computations leading to an equation for $X_0(\fp)$ are valid also in even characteristic. For example, for $q=2$ and $\fp=T^2+T+1$ 
	one obtains
	$$
	X_0(\fp)\colon \quad j \cdot x^{q^2} = (x-\fp)^{(q+1)}(x^2+x+\fp). 
	$$
\end{rem}

\begin{example}
	Now let $q=5$ and $\fp=T^2+T+2$. One computes that 
	\begin{align*}
		&\eta(s) = s^{-1}+s^3+(2T + 1)s^4+\fp s^5+\cdots \\ 
		& j(s) = -s^{-1} + (T^5 - T) - s^3 + (T^{25} + T^5 + 3T)s^4 + (4T^{30} + T^{26} + T^6 + 4T^2)s^5+\cdots.
	\end{align*}
	With the help of these expansions, one computes the 
	defining equation of $X_0(\fp)$:
	$$j = -\frac{(\eta^4 - (T + 3)(T^2 + T + 1)\fp \eta^3 + 2(T + 2)(T+4)\fp^2\eta^2 + (2T + 1)\fp^3 \eta - \fp^4)^6 (\eta^2 + (2T + 1)\eta + \fp)}{\eta^{q^2}}.$$
	Similar computations for $\fp = T^2+T+1$ lead to 
	$$j = -\frac{(\eta^4 - (T + 3)(T^2 + T + 2)\fp\eta^3 + 2T(T+1)\fp^2\eta^2 + (2T+1)\fp^3\eta - \fp^4)^6(\eta^2 + (2T + 1)\eta + \fp)}{\eta^{q^2}}.$$
	Thus, we conclude that in both cases 
	$
	X_0^1(\fp)$ is defined by $y^2=x(x^2+(2T+a)x+\fp)$. 
\end{example}

Based on the above examples, we propose that for odd $q$ and $\fp=T^2+aT+b$, the elliptic curve $X_0^1(\fp)$ is defined by 
\begin{equation}\label{eq-eqX_0^1(P)}
\boxed{y^2=x(x^2+(2T+a)x+\fp)}
\end{equation}

\begin{rem}
It is easy to see certain patterns in $-f(x)/(x^2+(2T+a)x+\fp)$, namely this is a polynomial of the form
$$
\left(x^{q-1}+b_{q-2}\fp x^{q-2}+b_{q-3}\fp^2x^{q-3}+\cdots+ b_0 \fp^{q-1}\right)^{q+1},
$$	
where $b_i\in \F_q[T]$, $\deg_T(b_i)=i$ for $0\leq i\leq q-2$, $b_0=-1$, $b_1=2T+a$. But what is the general formula for these polynomials  is  not clear to us. 
\end{rem}

The elliptic curve \eqref{eq-eqX_0^1(P)} is visibly defined over $F$. Moreover, as one verifies using Tate's algorithm, its 
conductor is $\fp\cdot \infty^2$ and its fibres at $\fp$ and $\infty$ are of types $I_2$ and $I_2^*$, respectively, in Kodaira's classification; 
cf. \cite[p. 354]{SilvermanII}.

\begin{lem}\label{lemExtEC}
	Let $E$ be an elliptic curve over $F$ with multiplicative reduction at $\fp$, bad reduction at $\infty$ of type $I_2^*$, and good reduction 
	at all other places of $F$. Such an $E$ is isomorphic over $\overline{\F}_q(T)$ to the elliptic curve defined by \eqref{eq-eqX_0^1(P)}. 
\end{lem}
\begin{proof}
	By assumption, the conductor of $E$ is $\fp\cdot \infty^2$, so it has degree $4$. Thus, $E$ is extremal over $\overline{\F}_q(T)$; see \cite[Prop. 4.2]{Ito}. 
	Such curves have been classified in characteristic $p\geq 5$ by W. Lang \cite{LangW} and $p=2, 3$ by Schweizer \cite{SchweizerExtremal}. 
	A convenient summary of Lang's result can be found in \cite[p. 720]{Ito}. 
	
	Now note that over  $\overline{\F}_q(T)$ the curve $E$ will have bad reduction at three places, two of which lie over $\fp$ and one lies over $\infty$. 
	The reduction at the places over $\fp$ will be multiplicative and the reduction at the place over $\infty$ must still 
	be additive as the conductor cannot have degree $< 4$. Thus, the type of $E$ over $\overline{\F}_q(T)$ is $(I_n, I_n, I_2^*)$. 
	Comparing with types of extremal elliptic curves in Lang's classification (see \cite[p. 720]{Ito} and also \cite[Prop. 4.2]{SchweizerExtremal}), 
	we see that only the type $(I_2, I_2, I_2^*)$ is possible, and there is a unique curve of this type, up to an isomorphism. Because the 
	curve defined by \eqref{eq-eqX_0^1(P)} has this type, the lemma follows. 
\end{proof}

By adapting an argument of Deligne and Rapoport for classical modular curves 
to Drinfeld modular curves, Gekeler showed in \cite{GekelerUber} that $X_0(\fp)$ has a model over $\Spec(A)$ which has semistable fibre 
at $\fp$ and smooth fibres at all other primes of $A$; here $\fp$ is an arbitrary prime of $A$.  
This heavily relies on the moduli interpretation of $X_0(\fp)$ and the work of Drinfeld \cite{Drinfeld}.  
Using similar methods, it should be possible to prove the same result also for $X_0^1(\fp)$, although the arguments will be quite technical. 
Assuming this, namely that $X_0^1(\fp)$ has a model over $\Spec(A)$ with semistable fibre 
at $\fp$ and smooth fibres at all other primes of $A$, we have the following: 

\begin{prop} Under the above assumption, $X_0^1(\fp)$ is isomorphic over $\overline{\F}_q(T)$ to the elliptic curve defined by \eqref{eq-eqX_0^1(P)} 
	when $\fp$ has degree $2$. 
\end{prop}
\begin{proof}
	Thanks to Lemma \ref{lemExtEC}, it is enough to show that the reduction type of the given model of $X_0^1(\fp)$ at $\infty$ is $I_2^*$.  
	Because the conductor of $X_0^1(\fp)$ must have degree at least $4$, the reduction at $\infty$ is additive. On the other hand, 
	$X_0^1(\fp)$  has rigid analytic uniformization over $\Fi$. Thus, it has a model over $\cO_\infty$ whose closed fibre 
	is the dual graph of $\G_0^1(\fp)\bs \cT$ without cusps. Since this graph consists of two edges 
	 joined at their terminal vertices and the stabilizers of the edges are $\pm 1$ (see Figure \ref{Fig3}), one concludes that $X_0^1(\fp)$ over $\Fi$ has 
	 multiplicative reduction and its component group has order $2$. (This is mostly a consequence of well-known facts about Mumford curves; 
	 cf. \cite[$\S$4.2]{PapikianAIF}.) Therefore, the fibre at $\infty$ of the given model of $X_0^1(\fp)$  over $F$ is of type $I_2^*$. 
\end{proof}


\section{Appendix} In this appendix we give pictures of some quotient graphs $\G_0(\fn)\bs \cT$ and $\G_0^1(\fn)\bs \cT$, and the 
natural covering map $\G_0^1(\fn)\bs \cT\to \G_0(\fn)\bs \cT$ between them. It is assumed that $q$ is odd. As in the main text, $\fp$ denotes a prime. 
The dotted arrows in the figures denote infinite half-lines (cusps); the cusps are labeled using representatives of orbits of $G_{d-1}$ acting on $\p^1(A/\fn)$.  
More specifically, $[\infty]\colonequals (1:0)$, $[a]\colonequals(a:1)$, $\begin{bmatrix} a\\ b\end{bmatrix}\colonequals(a:b)$ if $\gcd(b, \fn)\neq 1$. 


\begin{figure}[h]
	\scalebox{0.7}
	{
	\begin{tikzpicture}[thick, node distance=1.5cm, inner sep=.6mm, vertex/.style={circle, fill=black}]
			\node (a) at (0,0) 
		{
			\begin{tikzpicture}[thick, node distance=1.5cm, inner sep=.6mm, vertex/.style={circle, fill=black}]
				
				\node[vertex] (v1) {};
				\node[vertex] (v0) [right of=v1, node distance=2cm] {};
				\draw (v1) edge[bend right=70]  (v0)  (v1) edge[bend left=70] (v0); 
				
				\node[vertex] (0) [below right of=v0, node distance=0cm] {};
				\node[vertex] (1) [above right of=0] {};
				\node[vertex] (2) [right of=1] {};
				\node[] (3) [right of=2, label=right:${[\infty]}$] {};
				\node[vertex] (-1) [below right of=0] {};
				\node[vertex] (-2) [right of=-1] {};
				\node[] (-3) [right of=-2, label=right:${[0]}$] {};
				
				\node[vertex] (0a) [below left of=v1, node distance=0cm] {};
				\node[vertex] (1a) [above left of=0a] {};
				\node[vertex] (2a) [left of=1a] {};
				\node[] (3a) [left of=2a, label=left:${[\fp]}$] {};
				\node[vertex] (-1a) [below left of=0a] {};
				\node[vertex] (-2a) [left of=-1a] {};
				\node[] (-3a) [left of=-2a, label=left:${[\fp']}$] {};
				
				\path[]
				(-2) edge (-1) (-1) edge (0) (0) edge (1) (1) edge (2);
				
				\path[]
				(-2a) edge (-1a) (-1a) edge (0a) (1a) edge (0a) (2a) edge (1a);
				
				\draw [thick, ->, dashed, line width=0.5mm] (-2) -- (-3);
				\draw [thick, ->, dashed, line width=0.5mm] (2) -- (3);
				\draw [thick, ->, dashed, line width=0.5mm] (-2a) -- (-3a);
				\draw [thick, ->, dashed, line width=0.5mm] (2a) -- (3a);
				
			\end{tikzpicture}
		};
		
		\node (b) at (a.south) [anchor=north,yshift=-1cm]
		{
\begin{tikzpicture}[thick, node distance=1.5cm, inner sep=.6mm, vertex/.style={circle, fill=black}]
	\node[vertex] (v1) {};
	\node[vertex] (v0) [right of=v1, node distance=2cm] {};
	\draw (v1) edge (v0); 
	
	\node[vertex] (0) [below right of=v0, node distance=0cm] {};
	\node[vertex] (1) [above right of=0] {};
	\node[vertex] (2) [right of=1] {};
	\node[] (3) [right of=2, label=right:${[\infty]}$] {};
	\node[vertex] (-1) [below right of=0] {};
	\node[vertex] (-2) [right of=-1] {};
	\node[] (-3) [right of=-2, label=right:${[0]}$] {};
	
	\node[vertex] (0a) [below left of=v1, node distance=0cm] {};
	\node[vertex] (1a) [above left of=0a] {};
	\node[vertex] (2a) [left of=1a] {};
	\node[] (3a) [left of=2a, label=left:${[\fp]}$] {};
	\node[vertex] (-1a) [below left of=0a] {};
	\node[vertex] (-2a) [left of=-1a] {};
	\node[] (-3a) [left of=-2a, label=left:${[\fp']}$] {};
	
	\path[]
	(-2) edge (-1) (-1) edge (0) (0) edge (1) (1) edge (2);
	
	\path[]
	(-2a) edge (-1a) (-1a) edge (0a) (1a) edge (0a) (2a) edge (1a);
	
	\draw [thick, ->, dashed, line width=0.5mm] (-2) -- (-3);
	\draw [thick, ->, dashed, line width=0.5mm] (2) -- (3);
	\draw [thick, ->, dashed, line width=0.5mm] (-2a) -- (-3a);
	\draw [thick, ->, dashed, line width=0.5mm] (2a) -- (3a);

\end{tikzpicture}
	};
\draw [->] (a)--(b);
\end{tikzpicture}
}
\caption{$\Gamma_0^1(\fp\fp')\backslash \cT\To \Gamma_0(\fp\fp')\backslash\cT$ for $q$ odd, $\deg(\fp)=\deg(\fp')=1$, $\fp\neq \fp'$}\label{Fig2}
\end{figure}
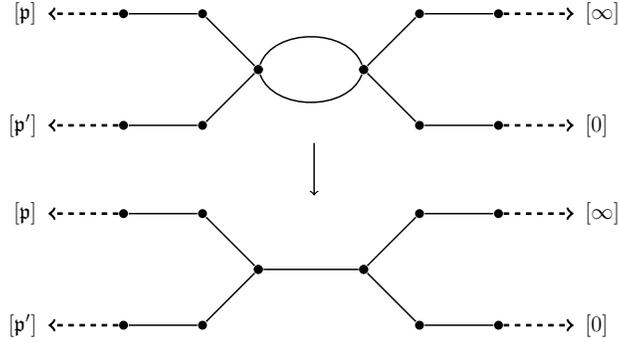


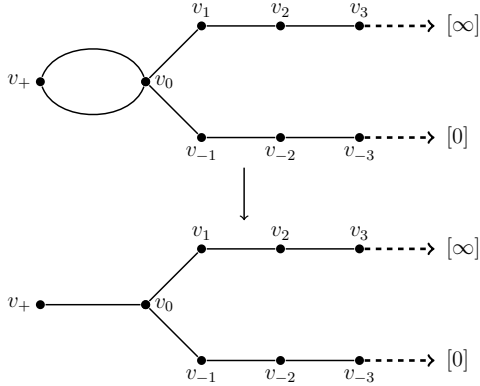
\begin{figure}[h]
	\scalebox{0.7}
	{
	\begin{tikzpicture}[thick, node distance=1.5cm, inner sep=.6mm, vertex/.style={circle, fill=black}]
		\node (a) at (0,0) 
		{
\begin{tikzpicture}[thick, node distance=1.5cm, inner sep=.6mm, vertex/.style={circle, fill=black}]
	
	\node[vertex] (1) [label=left:$v_{+}$] {};
	\node[vertex] (v0) [right of=1, node distance=2cm, label=right:$v_{0}$] {};
	\draw (1) edge[bend right=70]  (v0)  (1) edge[bend left=70] (v0); 
	
	\node[vertex] (0) [below right of=v0, node distance=0cm] {};
	\node[vertex] (1a) [above right of=0, label=above:$v_1$] {};
	\node[vertex] (2) [right of=1a, label=above:$v_2$] {};
	\node[vertex] (3) [right of=2, label=above:$v_3$] {};
	\node[] (4) [right of=3, label=right:${[\infty]}$] {};
	\node[vertex] (-1) [below right of=0, label=below:$v_{-1}$] {};
	\node[vertex] (-2) [right of=-1, label=below:$v_{-2}$] {};
	\node[vertex] (-3) [right of=-2, label=below:$v_{-3}$] {};
	\node[] (-4) [right of=-3, label=right:${[0]}$] {};
	
	\path[]
	(-3) edge (-2) (-2) edge (-1) (-1) edge (0) (0) edge (1a) (1a) edge (2) (2) edge (3);
	
	\draw [thick, ->, dashed, line width=0.5mm] (3) -- (4);
	\draw [thick, ->, dashed, line width=0.5mm] (-3) -- (-4);
	
\end{tikzpicture}
};

\node (b) at (a.south) [anchor=north,yshift=-1cm]
{
	\begin{tikzpicture}[thick, node distance=1.5cm, inner sep=.6mm, vertex/.style={circle, fill=black}]
		
		\node[vertex] (1) [label=left:$v_{+}$] {};
		\node[vertex] (v0) [right of=1, node distance=2cm, label=right:$v_{0}$] {};
		\draw (1) edge  (v0);
		
		\node[vertex] (0) [below right of=v0, node distance=0cm] {};
		\node[vertex] (1a) [above right of=0, label=above:$v_1$] {};
		\node[vertex] (2) [right of=1a, label=above:$v_2$] {};
		\node[vertex] (3) [right of=2, label=above:$v_3$] {};
		\node[] (4) [right of=3, label=right:${[\infty]}$] {};
		\node[vertex] (-1) [below right of=0, label=below:$v_{-1}$] {};
		\node[vertex] (-2) [right of=-1, label=below:$v_{-2}$] {};
		\node[vertex] (-3) [right of=-2, label=below:$v_{-3}$] {};
		\node[] (-4) [right of=-3, label=right:${[0]}$] {};
		
		\path[]
		(-3) edge (-2) (-2) edge (-1) (-1) edge (0) (0) edge (1a) (1a) edge (2) (2) edge (3);
		
		\draw [thick, ->, dashed, line width=0.5mm] (3) -- (4);
		\draw [thick, ->, dashed, line width=0.5mm] (-3) -- (-4);
		
	\end{tikzpicture}
	};
\draw [->] (a)--(b);
\end{tikzpicture}
}
\caption{$\Gamma_0^1(\mathfrak{p})\backslash \mathscr{T}\To\Gamma_0(\mathfrak{p})\backslash \mathscr{T}$ for $q$ odd and $\deg(\mathfrak{p}) = 2$}\label{Fig3}
\end{figure}


\begin{figure}[h]
	\scalebox{0.7}
	{
		\begin{tikzpicture}[thick, node distance=1.5cm, inner sep=.6mm, vertex/.style={circle, fill=black}]
			\node (a) at (0,0) 
			{
\begin{tikzpicture}[thick, node distance=1.5cm, inner sep=.6mm, vertex/.style={circle, fill=black}]
	
	\node[vertex] (0) {};
	\node[vertex] (1) [above right of=0] {};
	\node[vertex] (2) [right of=1] {};
	\node[] (3) [right of=2, label=right:${[\infty]}$] {};
	\node[vertex] (-1) [below right of=0] {};
	\node[vertex] (-2) [right of=-1] {};
	\node[] (-3) [right of=-2, label=right:${[0]}$] {};
	
	\node[vertex] (1a) [above left of=0] {};
	\node[vertex] (2a) [left of=1a] {};
	\node[] (3a) [left of=2a, label=left:${\begin{bmatrix}
			1\\ \mathfrak{p}
	\end{bmatrix}}$] {};
	\node[vertex] (-1a) [below left of=0] {};
	\node[vertex] (-2a) [left of=-1a] {};
	\node[] (-3a) [left of=-2a, label=left:${\begin{bmatrix}
			s\\ \mathfrak{p}
	\end{bmatrix}}$] {};
	
	\path[]
	(-2) edge (-1) (-1) edge (0) (0) edge (1) (1) edge (2) (-2a) edge (-1a) (-1a) edge (0) (1a) edge (0) (2a) edge (1a);
	
	\draw [thick, ->, dashed, line width=0.5mm] (2) -- (3);
	\draw [thick, ->, dashed, line width=0.5mm] (-2) -- (-3);
	\draw [thick, ->, dashed, line width=0.5mm] (2a) -- (3a);
	\draw [thick, ->, dashed, line width=0.5mm] (-2a) -- (-3a);
	
\end{tikzpicture}
	};
	
	\node (b) at (a.south) [anchor=north,yshift=-1cm]
	{
		\begin{tikzpicture}[thick, node distance=1.5cm, inner sep=.6mm, vertex/.style={circle, fill=black}]
			
			\node[vertex] (0) {};
			\node[vertex] (1) [above right of=0] {};
			\node[vertex] (2) [right of=1] {};
			\node[] (3) [right of=2, label=right:${[\infty]}$] {};
			\node[vertex] (-1) [below right of=0] {};
			\node[vertex] (-2) [right of=-1] {};
			\node[] (-3) [right of=-2, label=right:${[0]}$] {};
			
			\node[vertex] (1a) [left of=0] {};
            \node[vertex] (2a) [left of=1a] {};
			\node[] (3a) [left of=2a, label=left:${\begin{bmatrix}
					1\\ \mathfrak{p}
			\end{bmatrix}}$] {};
			
			\path[]
			(-2) edge (-1) (-1) edge (0) (0) edge (1) (1) edge (2) (1a) edge (0) (2a) edge (1a);
			
			\draw [thick, ->, dashed, line width=0.5mm] (2) -- (3);
			\draw [thick, ->, dashed, line width=0.5mm] (-2) -- (-3);
			\draw [thick, ->, dashed, line width=0.5mm] (2a) -- (3a);
			
		\end{tikzpicture}
	};
\draw [->] (a)--(b);
\end{tikzpicture}
}
\caption{$\Gamma_0^1(\fp^2)\backslash \mathscr{T}\To \Gamma_0(\fp^2)\backslash \mathscr{T}$ for $q$ odd, $\deg(\fp)=1$, $s\in \F_q^\times -  (\F_q^\times)^2$}\label{Fig4}
\end{figure}
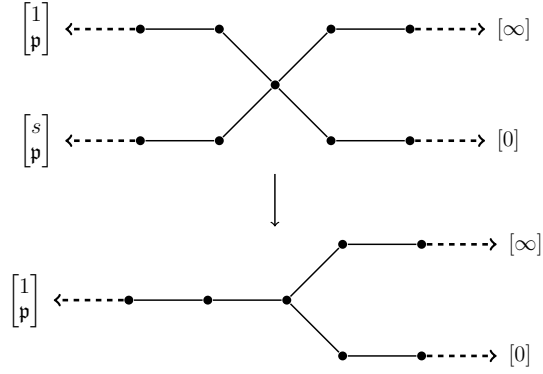

\newpage


\begin{figure}[h]
	\scalebox{0.5}
	{
		\begin{tikzpicture}[thick, node distance=1.5cm, inner sep=.6mm, vertex/.style={circle, fill=black}]
			\node (a) at (0,0) 
			{
\begin{tikzpicture}[thick, node distance=1.5cm, inner sep=.6mm, vertex/.style={circle, fill=black}]
	
	\node[vertex] (0) {};
	\node[vertex] (1) [left of=0, node distance=5cm] {};
	\node[vertex] (2) [right of=0, node distance=5cm] {};
	\node[] (1a) [left of=1, node distance=1cm, label=left:${[0]}$] {};
	\node[vertex] (base) [left of=1a, node distance=2cm, label=left:$v_{+}$] {};
	\node[] (2a) [right of=2, node distance=1cm, label=right:${[\infty]}$] {};
	\node[] (3b) [below of=1, node distance=5cm] {};
	\node[] (4b) [below of=2, node distance=5cm] {};
	\node[vertex] (3) [right of=3b, node distance=2cm] {};
	\node[vertex] (4) [left of=4b, node distance=2cm] {};
	\node[vertex] (5) [below of=0, node distance=1cm] {};
	\node[] (5a) [below of=5, node distance=0.7cm, label=below:$\begin{bmatrix}1\\ \mathfrak{p}\end{bmatrix}$] {};
	\node[vertex] (6) [below of=5, node distance=2cm] {};
	\node[] (6a) [below of=6, node distance=0.5cm] {};
	\node[vertex] (7) [below of=6, node distance=1cm] {};
	\node[] (7a) [below of=7, node distance=0.5cm, label=left:$\begin{bmatrix}T+1\\ \mathfrak{p}\end{bmatrix}$] {};
	\node[vertex] (8) [below of=7, node distance=1cm] {};
	\node[] (8a) [below of=8, node distance=0.5cm, label=right:$\begin{bmatrix}T+2\\ \mathfrak{p}\end{bmatrix}$] {};
	\node[vertex] (9) [below of=8, node distance=1cm] {};
	\node[] (9a) [below of=9, node distance=0.5cm, label=left:$\begin{bmatrix}T\\ \mathfrak{p}\end{bmatrix}$] {};
	
	\node[vertex] (up3) [above of=3, node distance=10cm] {};
	\node[vertex] (up4) [above of=4, node distance=10cm] {};
	\node[vertex] (up5) [above of=0, node distance=1cm] {};
	\node[] (up5a) [above of=up5, node distance=0.7cm, label=above:$\begin{bmatrix}-1\\ \mathfrak{p}\end{bmatrix}$] {};
	\node[vertex] (up6) [above of=up5, node distance=2cm] {};
	\node[] (up6a) [above of=up6, node distance=0.5cm] {};
	\node[vertex] (up7) [above of=up6, node distance=1cm] {};
	\node[] (up7a) [above of=up7, node distance=0.5cm, label=right:$\begin{bmatrix}-(T+1)\\ \mathfrak{p}\end{bmatrix}$] {};
	\node[vertex] (up8) [above of=up7, node distance=1cm] {};
	\node[] (up8a) [above of=up8, node distance=0.5cm, label=left:$\begin{bmatrix}-(T+2)\\ \mathfrak{p}\end{bmatrix}$] {};
	\node[vertex] (up9) [above of=up8, node distance=1cm] {};
	\node[] (up9a) [above of=up9, node distance=0.5cm, label=right:$\begin{bmatrix}-T\\ \mathfrak{p}\end{bmatrix}$] {};
	
	\draw [thick, ->, dashed, line width=1mm] (1) -- (1a);
	\draw [thick, ->, dashed, line width=1mm] (2) -- (2a);
	\draw (1) --  (0) --  (2);
	
	\draw (8) -- (3) -- (1) (2) -- (4) -- (8);
	\draw (3) -- (7) -- (4) -- (9) -- (3) (0) -- (5);
	\draw (5) edge[bend right=70]  (6) (5) edge[bend left=70]  (6);
	\draw (6) edge  (7) (6) edge[bend left=70]  (up8);
	\draw [thick, ->, dashed, line width=1mm] (5) -- (5a);
	\draw [thick, ->, dashed, line width=1mm] (7) -- (7a);
	\draw [thick, ->, dashed, line width=1mm] (8) -- (8a);
	\draw [thick, ->, dashed, line width=1mm] (9) -- (9a);
	
	\draw (up8) -- (up3) -- (1) (2) -- (up4) -- (up8);
	\draw (up3) -- (up7) -- (up4) -- (up9) -- (up3) (0) -- (up5);
	\draw (up5) edge[bend right=70]  (up6) (up5) edge[bend left=70]  (up6);
	\draw (base) edge[bend right=50]  (9) (base) edge[bend left=50]  (up9);
	\draw (up6) edge  (up7) (up6) edge[bend left=70]  (8);
	\draw [thick, ->, dashed, line width=1mm] (up5) -- (up5a);
	\draw [thick, ->, dashed, line width=1mm] (up7) -- (up7a);
	\draw [thick, ->, dashed, line width=1mm] (up8) -- (up8a);
	\draw [thick, ->, dashed, line width=1mm] (up9) -- (up9a);
\end{tikzpicture}
	};

\node (b) at (a.south) [anchor=north,yshift=-1cm]
{
	\begin{tikzpicture}[thick, node distance=1.5cm, inner sep=.6mm, vertex/.style={circle, fill=black}]
		
		\node[vertex] (0) {};
		\node[vertex] (1) [left of=0, node distance=3cm] {};
		\node[vertex] (2) [right of=0, node distance=3cm] {};
		\node[] (1a) [left of=1, node distance=1cm, label=left:${[0]}$] {};
		\node[] (2a) [right of=2, node distance=1cm, label=right:${[\infty]}$] {};
		\node[vertex] (3) [below of=1, node distance=5cm] {};
		\node[vertex] (4) [below of=2, node distance=5cm] {};
		\node[vertex] (5) [below of=0, node distance=1cm] {};
		\node[] (5a) [below of=5, node distance=0.7cm, label=below:$\begin{bmatrix}1\\ \mathfrak{p}\end{bmatrix}$] {};
		\node[vertex] (6) [below of=5, node distance=2cm] {};
		\node[] (6a) [below of=6, node distance=0.5cm, label=right:$\begin{bmatrix}T+1\\ \mathfrak{p}\end{bmatrix}$] {};
		\node[vertex] (7) [below of=6, node distance=1cm] {};
		\node[] (7a) [below of=7, node distance=0.5cm, label=left:$\begin{bmatrix}T+2\\ \mathfrak{p}\end{bmatrix}$] {};
		\node[vertex] (8) [below of=7, node distance=1cm] {};
		\node[] (8a) [below of=8, node distance=0.5cm, label=right:$\begin{bmatrix}T\\ \mathfrak{p}\end{bmatrix}$] {};
		\node[vertex] (9) [below of=8, node distance=1cm] {};
		\node[vertex] (10) [below of=9, node distance=1cm, label=right:$v_{+}$] {};
		
		\draw (8) -- (3) -- (1) --  (0) --  (2) -- (4) -- (8);
		\draw (3) -- (7) -- (4) -- (9) -- (3) (9) -- (10) (0) -- (5);
		\draw (5) edge[bend right=70]  (6) (5) edge[bend left=70]  (6);
		\draw (6) edge[bend right=70]  (7) (6) edge[bend right=70]  (8);
		\draw [thick, ->, dashed, line width=1mm] (1) -- (1a);
		\draw [thick, ->, dashed, line width=1mm] (2) -- (2a);
		\draw [thick, ->, dashed, line width=1mm] (5) -- (5a);
		\draw [thick, ->, dashed, line width=1mm] (8) -- (7a);
		\draw [thick, ->, dashed, line width=1mm] (7) -- (6a);
		\draw [thick, ->, dashed, line width=1mm] (9) -- (8a);
	\end{tikzpicture}
	
		};
	\draw [->] (a)--(b);
\end{tikzpicture}
}
\caption{$\Gamma_0(\mathfrak{p}^2)\backslash \mathscr{T}$, $\mathfrak{p} = T^2+1$, $q = 3$}\label{Fig5}
\end{figure}
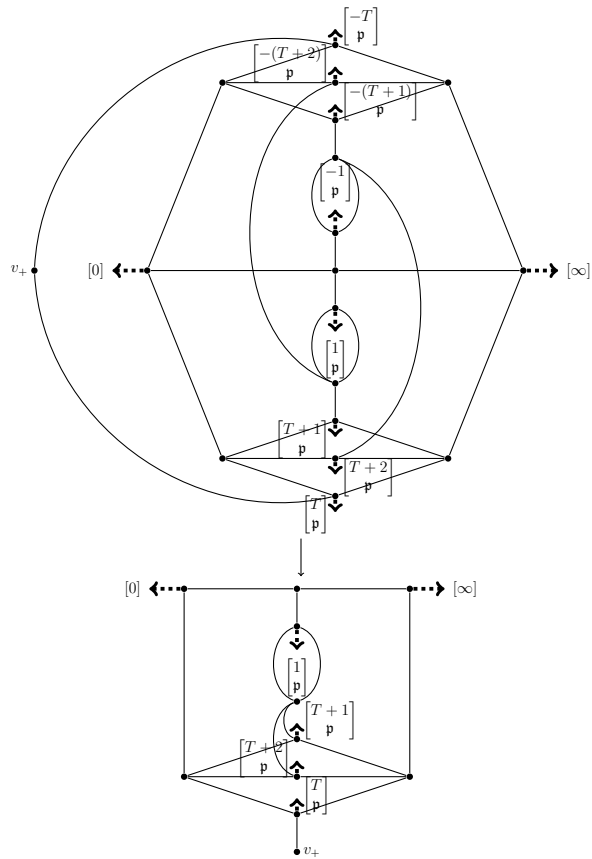


\newpage


\renewcommand{\bibliofont}{\normalsize}
\bibliographystyle{amsalpha}
\bibliography{bibliography.bib}

\end{document}